\documentclass[11pt,a4paper,openright,oneside]{article}
\usepackage{amsfonts, amsmath, amssymb,latexsym,amsthm, mathrsfs, enumerate}
\usepackage[english]{babel}
\usepackage{epsfig}
\usepackage{xcolor,cite}
\usepackage{mathtools}
\usepackage[hidelinks]{hyperref}
\usepackage[new]{old-arrows}
\usepackage{stackengine}
\usepackage{scalerel,wasysym}
\usepackage[normalem]{ulem}
\usepackage{stmaryrd}
\usepackage{thm-restate}
\usepackage{enumitem}

\makeatletter
\renewcommand\thmt@autorefsetup{%
  \@xa\def\csname\thmt@envname autorefname\@xa\endcsname\@xa{\thmt@thmname}%
}
\makeatother

\newcommand{\qedblack}{\hfill \ensuremath{\blacksquare}}

\stackMath

\newcommand{\Cay}{\mathrm{Cay}}
\newcommand{\End}{\mathrm{End}}
\newcommand{\Aut}{\mathrm{Aut}}

\usepackage[margin=1.2in]{geometry}
\usepackage{graphicx}
\usepackage{listings}
\usepackage[latin1]{inputenc}

\makeatletter
\providecommand\dotbigcup{\mathpalette\@barred\cdot}
\def\@barred#1#2{\ooalign{\hfil$#1\bigcup$\hfil\cr\hfil$#1#2$\hfil\cr}}

\makeatother

\numberwithin{equation}{section}
\newtheorem{teo}{Theorem}[section]
\newtheorem*{teo*}{Theorem}
\newtheorem*{prop*}{Proposition}
\newtheorem*{corol*}{Corollary}

\newtheorem{lemma}[teo]{Lemma}

\newtheorem{obs}[teo]{Observation}

 \theoremstyle{definition}
 
 \newtheorem{claim}{Claim}[teo]
 \newtheorem{fact}{Fact}

 \newtheorem{innercustomclaim}{Claim}
 \newenvironment{customclaim}[1]
  {\renewcommand\theinnercustomclaim{#1}\innercustomclaim}
  {\endinnercustomclaim}

\newtheorem{remark}[teo]{Remark}

\usepackage{fancyhdr}

\begin{document}
\title{On rigid regular graphs and a problem of Babai and Pultr}
\author{Kolja Knauer\footnote{Departament de Matem\`atiques i Inform\`atica, Universitat de Barcelona, Spain} \footnote{LIS, Aix-Marseille Universit\'e, CNRS, and Universit\'e de Toulon, Marseille, France}\and Gil Puig i Surroca\footnote{Universit\'e Paris-Dauphine, Universit\'e PSL, CNRS, LAMSADE, 75016, Paris, France}}

\maketitle

\begin{abstract}
A graph is \emph{rigid} if it only admits the identity endomorphism. We determine the minimum number of vertices of a rigid $d$-regular graph for every $d$. This extends results of Baron and Imrich [Acta Math.~Acad.~Sci.~Hung., 20 1969]. Two rigid graphs are further \emph{mutually rigid} if there is no homomorphism from one to the other. We show that for every $d\ge 3$ there exist infinitely many mutually rigid $d$-regular graphs of arbitrary odd girth $g\geq 7$. This provides strong positive answers to a question of van der Zypen from 2018 [\url{https://mathoverflow.net/q/296483}, \url{https://mathoverflow.net/q/321108}]. With the above construction as a building block, we show that every finite monoid is isomorphic to the endomorphism monoid of a regular graph. This solves a problem of Babai and Pultr [J. Comb.~Theory, Ser.~B, 1980]. 
\end{abstract}

\begin{center} 
\small{\textbf{Mathematics Subject Classifications:} 05C25,20M30,05C75,05C60 
 }
 
\small{\textbf{Keywords:} graph endomorphisms, monoids, rigidity, regular graphs}
\end{center} 

\section{Introduction}

~\hfill 
All graphs, groups, and monoids are considered to be finite. 

Frucht~\cite{Frucht1939} shows that every group is isomorphic to the \emph{automorphism group} $\Aut(G)$ of a graph $G$. Later, Hedrl\'in and Pultr~\cite{HP64,HP65} show  that every monoid $M$ is isomorphic to the \emph{endomorphism monoid} $\End(G)$ of a graph $G$. It is a natural and well-studied question what further assumptions can be made about $G$, while maintaining the above theorems valid, see the introduction of~\cite{knauer2023endomorphismuniversalitysparsegraph}. Here, we focus on $G$ being regular.

It is another classic result of Frucht~\cite{Frucht49} that every group is isomorphic to the automorphism group $\Aut(G)$ of a $3$-regular graph $G$ (even of a Hamiltonian $3$-regular $G$, see~\cite{BFKS79}). Soon after, Sabidussi~\cite{S57} showed that every group is isomorphic to the automorphism group $\Aut(G)$ of a $d$-regular graph $G$ for every $d\geq 3$. In particular, this applies to the trivial group $\{e\}$, i.e., there are $d$-regular \emph{asymmetric} graphs for every $d\geq 3$. Izbicki~\cite{Izb60} shows that for every $d\geq 3$ there exist infinitely many $d$-regular {asymmetric} graphs and Baron and Imrich~\cite{BI69} completely determine the function $\mu(d)$ that assigns to every $d\geq 3$ the smallest \emph{order}, i.e., number of vertices, of an asymmetric $d$-regular graph. Namely:

$$\mu(d)=\begin{cases}
12 & \text{if }d=3, \\
10 & \text{if }d=4,5, \\
11 & \text{if }d=6, \\
2\lceil\frac{d}{2}\rceil+4 & \text{if } d>6.\\
\end{cases}$$

Extending the result of Frucht~\cite{Frucht49}, Hell and Ne\v{s}et\v{r}il~\cite{HN73} show that every group is isomorphic to the endomorphism monoid of a $3$-regular graph. In particular, this applies to the trivial group $\{e\}$, i.e., there are $3$-regular \emph{rigid} graphs. Van der Zypen~\cite{vdZ18,vdZ19} asks for which $d$ rigid $d$-regular graphs exist and whether this holds for every large enough $d$. Godsil~\cite{God19} gives a beautiful partial answer to this via asymmetric block graphs of Steiner triple systems, which exist due to Babai~\cite{Bab80}. Using a result of Godsil and Royle~\cite{GR11} this yields (a finite set of strongly) $d$-regular rigid graphs for infinitely many $d$. 

In our first result we give a strong positive answer to van der Zypen's question, determining the order of the smallest instances:

\begin{restatable}{teo}{smallrigid}
\label{thm:smallrigid}
    For every $d\geq 3$ the smallest order of a rigid $d$-regular graph $\nu(d)$ behaves as follows:
    $$\nu(d)=\begin{cases}
14 & \text{if }d=3, \\
10 & \text{if }d=4,5, \\
11 & \text{if }d=6, \\
2\lceil\frac{d}{2}\rceil+4 & \text{if } d>6.\\
\end{cases}$$
    
\end{restatable}

We say that a family of graphs is \emph{mutually rigid}, if whenever there is a homomorphism $\varphi: H\to G$ for two of its members, then $H=G$ and $\varphi$ is the identity. Note that all the members of a mutually rigid family are rigid. The \emph{odd girth} of a graph is the length of its shortest odd cycle. Our second result gives a structural strong positive answer to van der Zypen's question:

\begin{restatable}{teo}{rigid}
\label{thm:regular_rigid_family}
    For every $d\geq 3$ and odd $g\geq 7$ there exist infinitely many mutually rigid $d$-regular graphs of odd girth $g$.
\end{restatable}

Extending the above-mentioned result of Hell and Ne\v{s}et\v{r}il~\cite{HN73} towards all monoids is impossible in a strong sense. Namely, Babai and Pultr~\cite{BP80} show that no graph class excluding a fixed topological minor can represent all monoids. In particular, there is no $d$ such that every monoid is isomorphic to the endomorphism monoid of a graph of maximum degree at most $d$. This prompted Babai and Pultr~\cite[Problem 2.3]{BP80} to ask whether every monoid is isomorphic to the endomorphism monoid of a regular graph. Using as a central ingredient Theorem~\ref{thm:regular_rigid_family} we are able to give a positive answer: 

\begin{restatable}{teo}{babaipultr}
\label{teo:babaipultr}
    For every monoid $M$ and odd $g\geq 7$ there exist infinitely many regular graphs $G$ of odd girth $g$ such that $\End(G)\cong M$.
\end{restatable}

This paper is structured as follows: Section~2 introduces some general terminology and notation; Section~3 is devoted to the proof of Theorem~\ref{thm:smallrigid}; Section~4 contains, in separate subsections, the tools to prove Theorems~\ref{thm:regular_rigid_family} and \ref{teo:babaipultr}, along with the proof of Theorem~\ref{thm:regular_rigid_family}; Section~5 contains the proof of Theorem~\ref{teo:babaipultr}; and Section 6 comments on related problems.

\section{Terminology and notation}
We denote by $[k]$ the set $\{1,\ldots,k\}$ of the first $k$ positive integers. 

Throughout the paper, graphs are undirected and simple (they have no loops or multiple edges), unless specified. Digraphs may have loops and anti-parallel arcs, but multiple parallel arcs are not allowed. This way, a digraph $D$ may be thought of as a relation on the set of vertices $V(D)$. Anti-reflexive symmetric relations correspond to (undirected) graphs, and anti-reflexive antisymmetric relations to \emph{oriented graphs}. A special case is a \emph{tournament}, which is an oriented complete graph, i.e., any two distinct vertices have exactly one arc between them. Another example is a \emph{directed cycle/path}, which is an oriented cycle/path whose vertices have at most one incoming and one outgoing arc. The symmetric closure (from the point of view of relations) of an oriented path is a \emph{bidirected path}. A \emph{homomorphism} between two (di)graphs $D$ and $D'$ is a mapping $\varphi:V(D)\rightarrow V(D')$ that sends arcs to arcs. When $\varphi$ is injective, $\varphi$ is called a \emph{monomorphism}, and when $D=D'$, $\varphi$ is an \emph{endomorphism}. If both hold simultaneously, $\varphi$ is an \emph{automorphism}. An \emph{isomorphism} is a bijective homomorphism that sends non-arcs to non-arcs. The set of homomorphisms between $D$ and $D'$ is denoted by $\mathrm{Hom}(D,D')$. The set of endomorphisms (automorphisms) of a (directed) graph $D$ is closed under composition and forms a monoid $\End(D)$ (group $\Aut(D)$). The \emph{indegree} and \emph{outdegree} of a vertex $v$ of $D$ is the number of incoming and outcoming arcs at $v$, respectively, and the \emph{total degree} of $v$ is the sum of its indegree and outdegree. The subgraph of $D$ induced by a set of vertices $W\subseteq V(D)$ is denoted by $D[W]$. We say that $D$ is \emph{connected} if its underlying undirected graph is connected.

A \emph{binary relational system} is a finite set $V$ together with a family of relations $A_i\subseteq V^2$, $i\in I$, where $I$ is a finite set. If we want to specify $I$, we call this a \emph{binary $I$-system}. Sometimes it is useful to understand a binary $I$-system as a digraph $D$ with $I$-coloured arcs, conveniently identifying $V$ with the set $V(D)$ of vertices of $D$, and $A_i$ with the set $A_i(D)$ of $i$-arcs (parallel arcs of different colours are allowed). A \emph{homomorphism} between two binary $I$-systems $D$ and $D'$ is a mapping $\varphi:V(D)\rightarrow V(D')$ that sends $i$-arcs to $i$-arcs, for every $i\in I$. Endomorphisms and automorphisms are defined in an analogous way, and the endomorphism monoid and the automorphism group of $D$ are denoted by $\End(D)$ and $\Aut(D)$ as in the case of graphs and digraphs. The \emph{indegree} (resp.~\emph{outdegree}) of a vertex $v$ of $D$ is the sum of the indegree (resp.~outdegree) of $v$ as a vertex of $(V,A_i)$ over all $i\in I$. The \emph{total degree} of $v$ is the sum of its indegree and outdegree, and is denoted by $\deg_D v$ (or just by $\deg v$). $D$ is an \emph{induced subsystem} of $D'=(V',\{A'_i\mid i\in I'\})$ if $I\subseteq I'$, $V\subseteq V'$, $A_i=A'_i\cap V^2$ for every $i\in I$, and $\emptyset=A'_i\cap V^2$ for every $i\in I'\backslash I$.

The neutral element of a monoid is often denoted by $e$. Given a monoid $M$ and a subset $C\subseteq M$, the \emph{coloured Cayley graph} of $M$ with respect to $C$ is the binary relational system $\Cay_{\mathrm{col}}(M,C)=(M,\{A_c\mid c\in C\})$, where $(u,v)\in A_c$ if and only if $uc=v$. The \emph{undirected, uncoloured Cayley graph} $\Cay(M,C)$ is just the graph $(M,E_C)$, where $\{u,v\}\in E_C$ if and only if there exists $c\in C$ such that $uc=v$.

\section{Small rigid regular graphs}\label{sec:small_rigid}
The results for small graphs obtained in this section have been mostly established by computer. For the exhaustive generation we use the code of Meringer~\cite{Mer99} and we implement checks for asymmetry, rigidity and similar properties in Sage~\cite{sagemath}. For convenience we provide the g6-codes of graphs that we consider interesting. 

The first such result is Observation~\ref{obs:cubic} below. We note that the non-existence of rigid $3$-regular graphs on less than $14$ vertices was already implicit in the literature: in~\cite{BCCS76}, it is found that there are only five asymmetric cubic graphs on less than $14$ vertices, and all of them have a triangle. By Brooks' theorem, they are $3$-colourable, so they are not rigid.


\begin{obs}\label{obs:cubic}
    There exists a rigid $3$-regular graph on $14$ vertices, see the left of Figure~\ref{fig:smallrigid}, and no smaller one.
\end{obs}

\begin{figure}[h]
\centering
\includegraphics[width=0.95\textwidth]{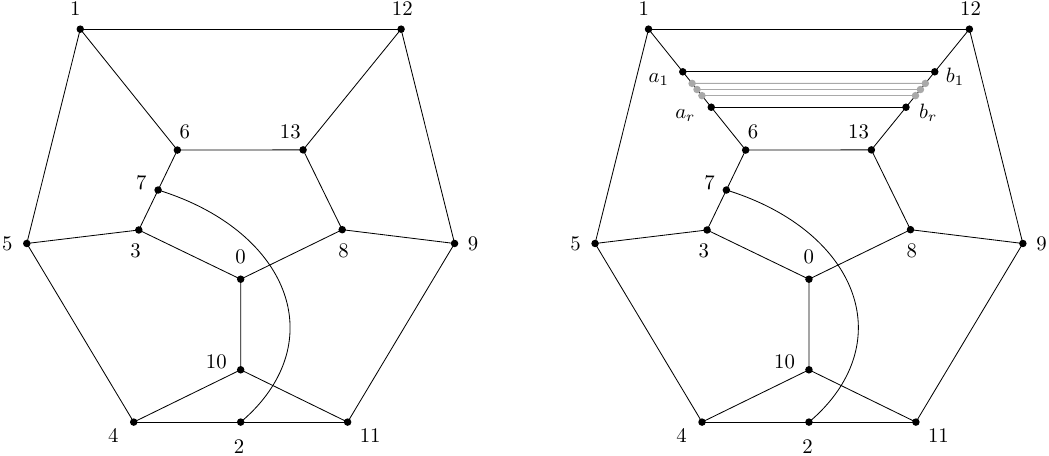}
\caption{A smallest $3$-regular rigid graph with g6-code \texttt{MCHY{@}e??KOCBOC?g\_} and the subdivided graph $G_r$ from Lemma~\ref{lem:cubic}.}
\label{fig:smallrigid}
\end{figure}

\begin{lemma}\label{lem:complement}
    If a graph $G$ is connected, asymmetric, $d$-regular, non-bipartite, and triangle-free, then the complement $\overline{G}$ is rigid and $n-d-1$-regular.
\end{lemma}
\begin{proof}
    It follows from the definitions that $\overline{G}$ is asymmetric, $n-d-1$-regular, and that its largest independent set is of size at most $2$. Suppose that $K$ is a subset of vertices that induces a complete subgraph in $\overline{G}$. Then $K$ is independent in $G$. Since $G$ is regular and connected, if the set of neighbours of $K$ in $G$ was of size at most $|K|$, then $G$ would be bipartite. Hence, the set of
vertices not adjacent to all of $K$ in $\overline{G}$ is larger than $|K|$. It follows from~\cite[Theorem 9]{HN92} that all endomorphisms of $\overline{G}$ are automorphisms.
\end{proof}

\begin{lemma}\label{lem:cubic}
    For every even $n\geq 14$ there exists a connected, asymmetric, $3$-regular, non-bipartite, and triangle-free graph $G$ on $n$ vertices.
\end{lemma}
\begin{proof}
Let $G_0$ be the graph on vertex set $\{0,1,\ldots,13\}$ depicted in the left of Figure~\ref{fig:smallrigid}. Let us denote $1$ by $a_0$ and $12$ by $b_0$. For $r\geq 1$, let $G_r$ be the graph obtained from $G_{r-1}$ as follows. First, we subdivide the edge
$\{a_{r-1},6\}$ with a new vertex $a_r$ and the edge $\{b_{r-1},13\}$ with a new vertex $b_r$, and then we 
add the edge $\{a_r,b_r\}$. See the right of Figure~\ref{fig:smallrigid} for an illustration.

Let $C$ be the $4$-cycle induced by $\{2,4,10,11\}$. Since $C$ is the unique $4$-cycle not sharing an edge with another $4$-cycle, any automorphism $\varphi\in\Aut(G_r)$ must send $C$ to itself. And since $\{2,11\}$ is the unique edge of $C$ not in a $5$-cycle, $\varphi$ sends $\{2,11\}$ to itself. Suppose that $\varphi(2)=11$, $\varphi(11)=2$. Then $\varphi(4)=10$ and $\varphi(10)=4$, so $\varphi(7)=9$ and $\varphi(0)=5$. But this is a contradiction, because $\varphi$ must send $3$ to a common neighbour of $9$ and $5$. Hence, $\varphi(2)=2$ and $\varphi(11)=11$, and $4$ and $10$ are also fixed. Then, $5$, $7$, $0$ and $9$ are fixed, $3$ and $8$ are also fixed, and $1$, $6$, $12$ and $13$ are also fixed, each time because they are neighbours of a fixed vertex with two fixed neighbours. Continuing like that, we end up seeing that $\varphi$ fixes all the vertices of $G_r$, i.e.~$G_r$ is asymmetric. And we are done: clearly, $G_r$ has $14+2r$ vertices and satisfies the rest of the properties. 
\end{proof}


\begin{lemma}\label{lem:circulant}
    Let $n\geq 9$ and $\varphi\in \Aut(\Cay(\mathbb{Z}_n,\{1,3\}))$ be an automporphism of the \emph{circulant} $\Cay(\mathbb{Z}_n,\{1,3\})$. If $\varphi$ has fixed points $0,1,2,3$, then $\varphi=\mathrm{id}$.
\end{lemma}
\begin{proof}
    Let $\varphi\in \Aut(\Cay(\mathbb{Z}_n,\{1,3\}))$ with $F=\{f_1, \ldots , f_k\}\subseteq \mathbb{Z}_n$ being a set of $k\geq 4$ consecutive fixed points. Then $f_1-1$ has neighbours $f_1$ and $f_3$, but since $n\geq 9$ their only other common neighbours are $f_2,f_4\in F$. Hence $f_1-1$ is also fixed by $\varphi$. This way  $\varphi=\mathrm{id}$.
\end{proof}

\begin{lemma}\label{lem:quartic}  
    For every even $n\geq 14$ there exists a connected, asymmetric, $4$-regular, non-bipartite, and triangle-free graph $G$ on $n$ vertices.
\end{lemma}
\begin{proof}
    Here are g6-codes for graphs on $n$ vertices satisfying the claim for every even $14\leq n\leq 22$, also see Figure~\ref{fig:quartics}: 
    \begin{itemize}[leftmargin=100pt,topsep=0pt,itemsep=0pt]
        \item[$n=14$:] \texttt{Ms\`{ }rQ\_gC?Q\_e?b?[\_}
        \item[$n=16$:] \texttt{Os\`{ }raOgCOW?O?O?L\_Do?\{}
        \item[$n=18$:] \texttt{Qs\`{ }raOgE?I?S?O?I?Ao?e?AK?FG}
        \item[$n=20$:] \texttt{Ss\`{ }raOgE?J?W?G?C\_A??Q?@g?Co?D\_?A[}
        \item[$n=22$:] \texttt{Us\`{ }AA?cG\`{ }AA\_CgCS@\`{ }?S?AO??\_O?gW?W\_?AH??XG}
    \end{itemize}

    \begin{figure}[ht]
        \centering
        \includegraphics[width=.6\linewidth]{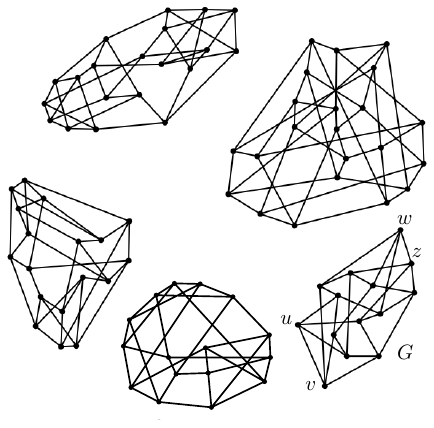}
        \caption{Small asymmetric, $4$-regular, non-bipartite, triangle-free graphs of even order.}
        \label{fig:quartics}
    \end{figure}
    
    One can check that the smallest graph $G$ from the above list has an induced matching  $e=\{u,v\},f=\{w,z\}$ of size $2$ such that every edge of $G\setminus\{e,f\}$ is contained in a $C_4$ and $G\setminus\{e,f\}$ is non-bipartite. See Figure~\ref{fig:quartics}.
    
    For a positive integer $k$ denote by $H_k=\Cay(\mathbb{Z}_k,\{1,3\})\setminus\{\{0,1\},\{2,3\}\}$ the graph obtained from the circulant $\Cay(\mathbb{Z}_k,\{1,3\})$ by removing two edges. Denote by $G_k$ the graph obtained from $H_k\cup G\setminus\{e,f\}$ by introducing the edges $\{u,0\}$, $\{v,3\}$, $\{w,1\}$, $\{z,2\}$. See Figure~\ref{fig:Gk} for an illustration.

\begin{figure}[ht]
    \centering
    \includegraphics[width=0.3\linewidth]{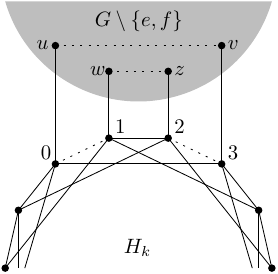}
    \caption{The construction of $G_k$.}
    \label{fig:Gk}
\end{figure}

    \begin{claim}\label{claim:Gk}
        \textit{For every $k\geq 10$, the graph $G_k$ is a connected, asymmetric, $4$-regular, non-bipartite, and triangle-free graph.}
    \end{claim}
    \noindent\textit{Proof of Claim~\ref{claim:Gk}.} All properties follow immediately except asymmetry, so assume that $\varphi \in \Aut(G_k)$. As mentioned above every edge of $G\setminus\{e,f\}$ is contained in a $C_4$. Further, every edge of $H_k$ except possibly $\{1,2\}$ and $\{0,3\}$ is contained in a $C_4$. None of the edges $\{u,0\}$, $\{v,3\}$, $\{w,1\}$, $\{z,2\}$ is  contained in a $C_4$. This is true because $\{e,f\}$ is an induced matching. Hence, $\varphi$ permutes the edges $\{u,0\}$, $\{v,3\}$, $\{w,1\}$, $\{z,2\}$, $\{1,2\}$, and $\{0,3\}$. But since these indeed induce two paths of length $3$, $\varphi$ permutes $u,v,w,z$ and separately $0,1,2,3$. Moreover $\{e, f\} = \{\varphi(e), \varphi(f)\}$. Thus, $\varphi$ restricts to $\varphi'\in \Aut(G)$. Hence, $\varphi'$ is the identity on the
vertices of $G$, and thus $\varphi$ also fixes each of $0,1,2,3$. By Lemma~\ref{lem:circulant} $\varphi$ restricts to the identity on $H_k$. We conclude that $\varphi$ is the identity.

    \qedblack

    Since $H_{k}$ exists for each even number of vertices $k\geq 10$, with Claim~\ref{claim:Gk} and the existence of the graph $G$ we can establish the statement of the lemma for every even $n\geq 24$.
\end{proof}

\smallrigid*
\begin{proof}
The result follows from Observation~\ref{obs:cubic} for $d=3$. The lower bound follows by $\mu(d)\leq \nu(d)$ and~\cite{BI69} for $d\geq 4$. For $4\leq d\leq 8$, we find rigid graphs of the claimed order by computer. Here is a list and also see Figure~\ref{fig:regular}:
\begin{itemize}[leftmargin=100pt,topsep=0pt,itemsep=0pt]
    \item[$d=4$:] \texttt{I\}hP?sM@w}
    \item[$d=5$:] \texttt{I\}qr@s]Bw}
    \item[$d=6$:] \texttt{J\textasciitilde  zcqgjDw\^\_}
    \item[$d=7$:] \texttt{K\textasciitilde\textasciitilde edXUHwv\`{}\textasciitilde}
    \item[$d=8$:] \texttt{K\textasciitilde\textasciitilde vUefRxzb\textasciitilde}
\end{itemize}

\begin{figure}[ht]
    \centering
    \includegraphics[width=0.6\linewidth]{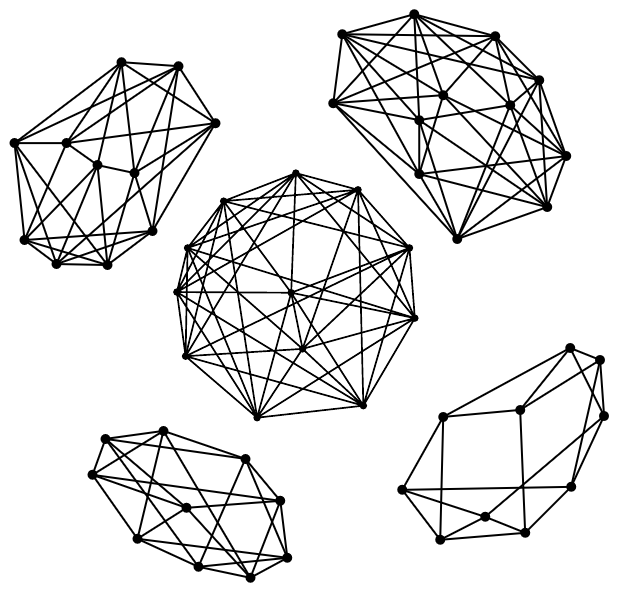}
    \caption{Smallest rigid $d$-regular graphs for $4\leq d\leq 8$.}
    \label{fig:regular}
\end{figure}

For the remaining values $d>8$ if $d$ is even take an asymmetric, $3$-regular, non-bipartite, and triangle-free graph $G$ on $n=d+4$ vertices with Lemma~\ref{lem:cubic}. If $d$ is odd, then take an asymmetric, $4$-regular, non-bipartite, and triangle-free graph $G$ on $n=d+5$ vertices, with Lemma~\ref{lem:quartic}. By Lemma~\ref{lem:complement} the graph $\overline{G}$ satisfies the hypothesis of the theorem.
\end{proof}

\section{Many rigid regular graphs}

Here and in Section~5, the strategies towards Theorems~\ref{thm:regular_rigid_family} and~\ref{teo:babaipultr} both follow the same principle: first, one solves a version of the problem for digraphs (or for binary relational systems), and then, one tries to obtain graphs with similar properties. The first step is performed in Section~\ref{sec:directed_gadgets}, and the transition from binary relational systems to graphs, in Section~\ref{sec:sip}. Essentially, this transition consists in replacing arcs by undirected gadgets. The gadgets for the case $d=3$ are constructed in Section~\ref{sec:tiling_factors}. The gadgets for the higher degree cases are constructed in Section~\ref{sec:mutually_rigid}, using a graph product described in Section~\ref{sect:cartesian}. These gadgets are then used to prove Theorem~\ref{thm:regular_rigid_family}, at the end of the section.

\subsection{Rigid digraphs with restrictions on the degrees}\label{sec:directed_gadgets}

A \emph{transitive tournament} $T$ is the transitive closure of a directed path. We denote by $v_i(T)$ the $i$th vertex of that path ($v_1(T)$ is the source itself).


The following is inspired by a construction of~\cite{HN73} and illustrated in Figure~\ref{fig:lemma}. Let $d,\ell\in\mathbb Z^+$ with $d\geq 3$ and let $f:[\ell]\times([d-1]\backslash\{1\})\rightarrow\{+,0,-\}$ be a mapping. We define a directed graph $S(d,\ell,f)$ as follows. Consider $2\ell$ copies of the transitive tournament on $d$ vertices, that we refer to as $T_+^1,\ldots,T_+^{\ell},T_-^1,\ldots,T_-^{\ell}$, two copies $T_l$ and $T_r$ of the transitive tournament on $d-1$ vertices, and four vertices $s_l,t_l,s_r,t_r$. We assume that these vertices, together with all the vertices of the tournaments, are pairwise distinct. Now we set:
\begin{align*}
V(S(d,\ell,f))\,=\;&\bigcup_{i=1}^{\ell} (V(T_+^i)\cup V(T_-^i))\,\cup\, V(T_r)\,\cup\, V(T_l)\,\cup\,\{s_r,t_r,s_l,t_l\},\\
A(S(d,\ell,f))\,=\;&\bigcup_{i=1}^{\ell} (A(T_+^i)\cup A(T_-^i))\,\cup\, A(T_r)\,\cup\, A(T_l) \\ 
&\cup\,\{(s_l,x),(x,t_l)\mid x\in V(T_l)\}\cup\{(s_r,x),(x,t_r)\mid x\in V(T_r)\} \ \phantom{\bigcup_{i=1}^{\ell}}\\
&\cup\,\{(v_d(T_+^1),s_l),(t_l,v_1(T_-^1)),(v_d(T_-^{\ell}),s_r),(t_r,v_1(T_+^{\ell}))\} \ \phantom{\bigcup_{i=1}}\\
&\cup\,\bigcup_{i=1}^{\ell-1}\{(v_d(T_+^{i+1}),v_1(T_+^i)),(v_d(T_-^i),v_1(T_-^{i+1}))\} \\
&\cup\,\bigcup_{i=1}^{\ell}\{(v_j(T_+^i),v_{d+1-j}(T_-^i))\mid 2\leq j\leq d-1,\ f(i,j)=-\} \\
&\cup\,\bigcup_{i=1}^{\ell}\{(v_{d+1-j}(T_-^i),v_{j}(T_+^i))\mid 2\leq j\leq d-1,\ f(i,j)=+\}.
\end{align*}

\begin{figure}[h]
\centering
\includegraphics[scale=1]{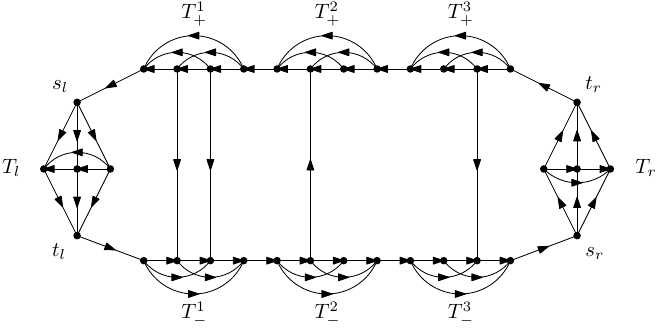}
\caption[A depiction of $S(4,3,f)$]{$S(4,3,f)$ for $f$ given by  
{\small
\begin{tabular}{c|c|c|c|c|c|c}
$(i,j)$ & $(1,3)$ & $(1,2)$ & $(2,3)$ & $(2,2)$ & $(3,3)$ & $(3,2)$ \\\hline
$f(i,j)$ & $-$ & $-$ & $+$ & $0$ & $0$ & $-$ \\
\end{tabular}}.}
\label{fig:lemma}
\end{figure}

Note that the total degree of every vertex of $S(d,\ell,f)$ is $d$, except for the vertices of the form $v_j(T^i_+)$ and $v_{d+1-j}(T^i_-)$ with $1\leq i\leq\ell$, $2\leq j\leq d-1$ and $f(i,j)=0$, which have total degree $d-1$.

When we consider various digraphs of the form $S(d,\ell,f)$ together (not necessarily with the same parameters $d,\ell,f$), and $D$ is one of them, we denote by $T^1_+(D),\ldots,T^{\ell}_+(D)$, $T^1_-(D),\ldots,T^{\ell}_-(D)$, $T_l(D)$, $T_r(D)$, $s_l(D)$, $t_l(D)$, $s_r(D)$, $t_r(D)$ the corresponding instances of $T^1_+,\ldots,T^{\ell}_+$, $T^1_-,\ldots,T^{\ell}_-$, $T_l$, $T_r$, $s_l$, $t_l$, $s_r$, $t_r$ in $D$.

Let us define an order $\leq$ on the set of mappings from $[\ell]\times([d-1]\backslash\{1\})$ to $\{+,0,-\}$ as follows. If $f,f'$ are two such mappings, then $f\leq f'$ if and only if, for every $(i,j)\in[\ell]\times([d-1]\backslash\{1\})$, either $f(i,j)=0$ or $f(i,j)=f'(i,j)$. We also define a new mapping $\mathrm{inv}f:[\ell]\times([d-1]\backslash\{1\})\rightarrow\{+,0,-\}$ from $f$ by setting $(\mathrm{inv}f)(i,j)=-f(\ell+1-i,d+1-j)$.

\begin{lemma}\label{lem:sausages} Let $d,\ell,\ell'\in\mathbb Z^+$ with $d\geq 3$ and let $f:[\ell]\times([d-1]\backslash\{1\})\rightarrow\{+,0,-\}$ and $f':[\ell']\times([d-1]\backslash\{1\})\rightarrow\{+,0,-\}$. Then, $$|\mathrm{Hom}(S(d,\ell,f),S(d,\ell',f'))|=\begin{cases}2 &\text{if $\ell=\ell'$, $f\leq f'$ and $\mathrm{inv}f\leq f'$} \\ 1 &\text{otherwise, if $\ell=\ell'$ and either $f\leq f'$ or $\mathrm{inv}f\leq f'$} \\ 0 &\text{in the rest of the cases.}\end{cases}$$
\end{lemma}
\begin{proof} Put $D=S(d,\ell,f)$, $D'=S(d,\ell',f')$. Due to the absence of loops, any homomorphism $\varphi:D\rightarrow D'$ must map tournaments of $D$ to tournaments of $D'$ of the same order. Since all vertices of $T_l(D)$ have a common out-neighbour, the tournament induced by $\{s_l(D)\}\cup V(T_l(D))$ must be mapped to $\{s_l(D')\}\cup V(T_l(D'))$ or to $\{s_r(D')\}\cup V(T_r(D'))$ in the unique possible way. In particular, $\varphi(s_l(D))\in\{s_l(D'),s_r(D')\}$, and similarly $\varphi(t_r(D))\in\{t_l(D'),t_r(D')\}$.

Suppose that $\varphi(s_l(D))=s_l(D')$. Then this forces $\varphi(V(T_+^1(D)))=V(T_+^1(D'))$ in the unique possible way, and recursively $\varphi(V(T_+^i(D)))=V(T_+^i(D'))$ in the unique possible way, for $2\leq i\leq\ell'$: since $\varphi(t_r(D))\in\{t_l(D'),t_r(D')\}$ we know that $\ell\geq\ell'$. Suppose for a contradiction that $\ell>\ell'$. Let $x$ be the in-neighbour of $v_1(T_+^{\ell'+1}(D))$; we know that the closed in-neighbourhood of $x$ is a transitive tournament of order $d\geq 3$. By assumption, $\varphi(V(T_+^{\ell'+1}(D)))=\{t_r(D')\}\cup V(T_r(D'))$ in the unique possible way, so $\varphi(x)=s_r(D')$. But the closed in-neighbourhood of $s_r(D')$ is just an arc, the desired contradiction. Hence $\ell=\ell'$, and $\varphi(\{t_r(D),s_r(D)\}\cup V(T_r(D)))=\{t_r(D'),s_r(D')\}\cup V(T_r(D'))$ in the unique possible way. It follows that $\varphi(v_j(T_-^i(D)))=v_j(T_-^i(D'))$ for every $1\leq i\leq\ell$ and $1\leq j\leq d$, so, if it exists, $\varphi$ is completely determined. And $\varphi$ exists if and only if $f\leq f'$.

In the case that $\varphi(s_l(D))=s_r(D')$, a symmetric argument shows that there exists at most one such $\varphi$, and that $\varphi$ exists if and only if $\ell=\ell'$ and $\mathrm{inv}f\leq f'$.
\end{proof}

With Lemma~\ref{lem:sausages} one can construct digraphs with certain properties that will be needed further ahead. More precisely, one can set 
 $\mathcal F_1(d)=\{S(d,\ell,+_{d,\ell})\mid\ell\geq 1\}$, where $+_{d,\ell}:[\ell]\times([d-1]\backslash\{1\})\rightarrow\{+\}$ is the constant positive mapping, and $\mathcal F_2(d)=\{S(d,\ell,+'_{d,\ell})\mid\ell\geq 2\}$, where $+'_{d,\ell}:[\ell]\times([d-1]\backslash\{1\})\rightarrow\{+,0\}$ sends $(1,d-1)$ to $0$ and everything else to $+$. We obtain:

\begin{remark}\label{rem:sausages}
    For each $d\geq 3$, two infinite families $\mathcal F_1(d)$ and $\mathcal F_2(d)$ of mutually rigid, connected, oriented graphs of minimum indegree and outdegree at least 1, satisfying the following properties.
\begin{itemize}
   \item[(1)] For every $D\in\mathcal F_1(d)\cup\mathcal F_2(d)$, every vertex of $D$ is on an oriented triangle;
   \item[(2)] for every $D\in\mathcal F_1(d)$, all vertices of $D$ have total degree $d$;
   \item[(3)] for every $D\in\mathcal F_2(d)$, there are two vertices $u,v$ of $D$ that have total degree $d-1$ while the rest have total degree $d$; moreover, any oriented path between $u$ and $v$ has length at least $3$.
\end{itemize}
\end{remark}

\subsection{The \v s\'ip product}\label{sec:sip}

This section covers a standard technique, in a form that suits our purposes. See~\cite{HN73,HN05,knauer2023endomorphismuniversalitysparsegraph} for three of many examples of its use, or~\cite[Section 4.4]{HN04} for a more detailed account. The idea is the following: given a class of binary relational systems, we aim to replace the arcs by undirected gadgets, in a way that the resulting graph class exhibits the same behaviour from the point of view of homomorphisms. This problem has a natural formulation in the language of categories (we refer the unfamiliar reader to~\cite{AHS2004} or~\cite{PT80} for the basic definitions missing here). Indeed, the category-theoretical framework is interesting in its own right, and has been amply explored, see~\cite[Section 4.9 and prec.]{HN04} and~\cite{PT80}.

An \emph{indicator} (also called a \emph{\v s\'ip}) is a digraph $S$ with a distinguished pair of vertices $(\mathrm{in}\,S,\mathrm{out}\,S)$ with $\mathrm{in}\,S\neq\mathrm{out}\,S$. A \emph{$d$-indicator} is an anti-reflexive symmetric indicator (i.e.~a graph) such that $\deg(\mathrm{in}\,S)=\deg(\mathrm{out}\,S)=d-1$ and $\deg(v)=d$ for every other vertex $v$. Consider a binary $[k]$-system $D$ and a $k$-tuple of indicators $\mathbf S=(S_1,\ldots,S_k)$. For every $i\in[k]$ and $a\in A_i(D)$, let $S_{i,a}$ be a copy of $S_i$. Assume that the vertex sets $V(S_{i,a})$, together with $V(D)$ and the sets $V(S_i)$, are pairwise disjoint. For every vertex $x$ of $S_i$ or of $S_{i,a}$, let $x_{a'}$ be its corresponding clone in $S_{i,a'}$. We define the \emph{\v s\'ip product} $D*\mathbf S$ to be the digraph with vertex set 
$$V(D)\,\cup\underset{\substack{i\in[k] \\ a\in A_i(D)}}{\bigcup}V(S_{i,a})$$ 
and arc set 
$$\{(x,\mathrm{in}\,S_{i,(x,y)}),(\mathrm{in}\,S_{i,(x,y)},x),(\mathrm{out}\,S_{i,(x,y)},y),(y,\mathrm{out}\,S_{i,(x,y)})\mid i\in[k],\ (x,y)\in A_i(D)\}$$$$\cup\underset{\substack{i\in[k] \\ a\in A_i(D)}}{\bigcup} A(S_{i,a}).$$
When $\mathbf S=(S_1)$ is a $1$-tuple, we also write $D*S_1$ for $D*\mathbf S$. See Figure~\ref{fig:sip_products} for an illustration. Note that if all indicators $S_1,\ldots,S_k$ are graphs, then $D*\mathbf S$ is also a graph. The following observations will be useful for later.

\begin{obs}\label{obs:odd_girth_sip_product} Let $D$ be a binary $[k]$-system and $\mathbf S=(S_1,\ldots,S_k)$ a $k$-tuple of anti-reflexive indicators.
\vspace{-1mm}
\begin{enumerate}[label=(\roman*)]
   \item\label{enum:odd_girth_sip_product_2} The length of any (oriented) odd cycle $C$ of $D*\mathbf S$ is at least
   \[\begin{cases}
   \min_{1\leq i\leq k}\mathrm{oddgirth}\,\underline{S_i} & \text{if $C$ has no vertex of $D$} \\
   \min_{1\leq i\leq k}\mathrm{odddist}_{\underline{S_i}}(\mathrm{in}\,S_i,\mathrm{out}\,S_i)+2 & \text{otherwise,}
   \end{cases}\]
   where $\underline{S_i}$ is the symmetric closure of $S_i$ and $\mathrm{odddist}_{\underline{S_i}}(\mathrm{in}\,S_i,\mathrm{out}\,S_i)$ is the minimum odd length of a path of $\underline{S_i}$ between $\mathrm{in}\,S_i$ and $\mathrm{out}\,S_i$. (Indeed, in the latter case, if $C$ has precisely $r\geq 1$ vertices of $D$, then $\mathrm{length}(C)=\sum_{j=1}^r\mathrm{length}(P_j)+2r$, where $P_j$ is a path of some $\underline{S_{i_j}}$ between $\mathrm{in}\,S_{i_j}$ and $\mathrm{out}\,S_{i_j}$, and one of these paths has to be of odd length.)
    \item\label{enum:odd_girth_sip_product_1} If $S_1,\ldots,S_k$ are symmetric and $v$ is a vertex of $D$ of total degree $d$, then $v$ has degree $d$ as a vertex of (the graph) $D*\mathbf S$. In particular, if every vertex of $D$ has the same total degree $d$, and $S_1,\ldots,S_k$ are $d$-indicators, then $D*\mathbf S$ is a $d$-regular graph.
\end{enumerate}
\end{obs}

\begin{figure}[h]
\centering
\includegraphics[scale=0.83]{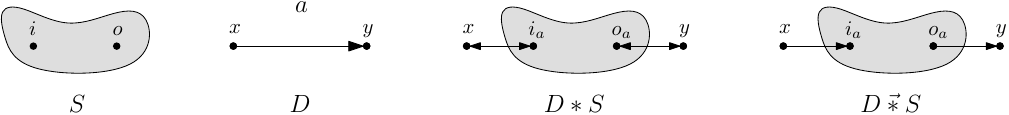}
\caption{An indicator $S$ with distinguished pair of vertices $(\mathrm{in}\,S,\mathrm{out}\,S)=(i,o)$, an arc $a=(x,y)$ in a binary $\{1\}$-system (i.e.~a digraph) $D$, and the gadgets replacing $a$ in $D\mathbin * S$ and in $D\mathbin{\vec *} S$.}
\label{fig:sip_products}
\end{figure}

We will also use a second version of the \v s\'ip product: $D\mathbin{\vec *}\mathbf S$ denotes the digraph with vertex set $V(D*\mathbf S)$ and arc set 
$$\{(x,\mathrm{in}\,S_{i,(x,y)}),(\mathrm{out}\,S_{i,(x,y)},y)\mid i\in[k],\ (x,y)\in A_i(D)\}\ \cup\underset{\substack{i\in[k] \\ a\in A_i(D)}}{\bigcup} A(S_{i,a}).$$
In this case, when all indicators $S_1,\ldots,S_k$ are oriented graphs, $D\mathbin{\vec *}\mathbf S$ is also an oriented graph. (See Figure~\ref{fig:sip_products}.)

The class of all binary $[k]$-systems, together with all homomorphisms between them, forms a category; we are going to call it $k\text{-}\mathsf{Systems}$. Let $k\text{-}\mathsf{Systems}'$ denote the subcategory of $k\text{-}\mathsf{Systems}$ consisting in all binary $[k]$-systems with minimum indegree and outdegree at least $1$, and all homomorphisms between them. Given a $k$-tuple of indicators $\mathbf S$, let $k\text{-}\mathsf{Systems}'*\mathbf S$ (resp.~$k\text{-}\mathsf{Systems}'\mathbin{\vec *}\mathbf S$) be the category consisting of all the objects of the form $D*\mathbf S$ (resp.~$D\mathbin{\vec *}\mathbf S$), where $D$ is an object of $k\text{-}\mathsf{Systems}'$, and all homomorphisms between them.

Recall that two categories $A$ and $B$ are \emph{equivalent} if there exists a functor $F:A\rightarrow B$ which is full, faithful, and satisfies that, for each object $b$ of $B$, there is an object $a$ of $A$ such that $F(a)$ and $b$ are isomorphic. 

\begin{lemma}\label{lem:sip_product_and_categories} Let $k$ be a positive integer and $\mathbf S=(S_1,\ldots,S_k)$ a $k$-tuple of indicators. Assume that
\vspace{-1mm}
\begin{enumerate}[label=(\roman*)]
   \item $S_1,\ldots,S_k$ are connected and mutually rigid, that
   \item\label{enum:sip_product_and_categories_odd_girth} there is an odd positive integer $g$ such that every $S_1,\ldots,S_k$ has an oriented cycle of length $g$ and no shorter oriented cycles of odd length, and that for every $i\in[k]$
   \item\label{enum:sip_product_and_categories_odd_cycles_everywhere} every vertex of $S_i$ is on an oriented cycle of length $g$;
   \item\label{enum:sip_product_and_categories_bidirected_path} the length of any bidirected path between $\mathrm{out}\, S_i$ and $\mathrm{in}\, S_i$ is at least $3$;
   \item\label{enum:sip_product_and_categories_oriented_path} any oriented path between $\mathrm{out}\, S_i$ and $\mathrm{in}\, S_i$ of odd length is of length at least $g$.
\end{enumerate}
\vspace{-1mm}
Then, $k\text{-}\mathsf{Systems}'$ and $k\text{-}\mathsf{Systems}'*\mathbf S$ are equivalent categories; and if (iv) is replaced by
\vspace{-1mm}
\begin{itemize}
   \item[(iv')] the length of any directed path from $\mathrm{out}\, S_i$ to $\mathrm{in}\, S_i$ is at least $3$,
\end{itemize}
\vspace{-1mm}
then $k\text{-}\mathsf{Systems}'$ and $k\text{-}\mathsf{Systems}'\mathbin{\vec *}\mathbf S$ are equivalent categories.
\end{lemma}
\begin{proof} We prove the first part of the statement; the proof works for the second part as well. Let $D$ and $D'$ be binary $[k]$-systems with minimum indegree and outdegree at least $1$, and let $\psi\in\mathrm{Hom}(D*\mathbf S,D'*\mathbf S)$. By \ref{enum:sip_product_and_categories_odd_girth}, \ref{enum:sip_product_and_categories_oriented_path}, Observation~\ref{obs:odd_girth_sip_product}\ref{enum:odd_girth_sip_product_2} and the fact that $D$ has at least one arc, shortest oriented odd cycles of $D*\mathbf S$ have length $g$. Moreover, there is no vertex of $V(D)$ on such a cycle, while the rest of vertices of $D*\mathbf S$ are, by~\ref{enum:sip_product_and_categories_odd_cycles_everywhere}. Thus $\psi(V(S_{i,a}))\cap V(D')=\emptyset$ for every $i\in[k]$ and $a\in A_i(D)$. Furthermore, by connectivity, $\psi(V(S_{i,a}))\subseteq V(S_{i',a'})$ for some $i'\in[k]$ and $a'\in A_{i'}(D')$. By the mutual rigidity, $i=i'$ and $\psi(x_a)=x_{a'}$ for every vertex $x$ of $S_i$. Let us see now that $\psi(V(D))\subseteq V(D')$. Suppose for a contradiction that $\psi(x)\in V(S_{i,a})$ for some $x\in V(D)$, $i\in[k]$ and $a\in A_i(D')$. The minimum in- and outdegree condition on $D$ says that there exist $i^+,i^-\in[k]$ and $y^+,y^-\in V(D)$ with $(y^+,x)\in A_{i^+}(D)$ and $(x,y^-)\in A_{i^-}(D)$. Then $\psi(V(S_{i^+,(y^+,x)})),\psi(V(S_{i^-,(x,y^-)}))\subseteq V(S_{i,a})$, so $i=i^+=i^-$, $\psi((\mathrm{out}\,S_i)_{(y^+,x)})=(\mathrm{out}\,S_i)_a$ and $\psi((\mathrm{in}\,S_i)_{(x,y^-)})=(\mathrm{in}\,S_i)_a$. But then there is a bidirected path of length $2$ between these two vertices, contradicting~\ref{enum:sip_product_and_categories_bidirected_path}.

Given $i\in[k]$ and $a=(x,y)\in A_i(D)$, let $(x',y')\in A_i(D')$ such that $\psi(S_{i,(x,y)})=S_{i,(x',y')}$. It is now clear that $\psi((x,y))=(x',y')$. In particular, this shows that $\psi\vert_{V(D)}\in\mathrm{Hom}(D,D')$, and implies that $\psi(x_a)=x_{\psi(a)}$ for each $x\in V(S_i)$. Now, for each $\varphi\in\mathrm{Hom}(D,D')$, let $\varphi*\mathbf S:V(D*\mathbf S)\rightarrow V(D'*\mathbf S)$ be the mapping defined by $$(\varphi*\mathbf S)(x)=\begin{cases}\varphi(x) &\text{if $x\in V(D)$} \\ x_{\varphi(a)} &\text{if $x\in V(S_{i,a})$ for some $i\in[k]$ and $a\in A_i(D)$.}\end{cases}$$ Note that $\varphi*\mathbf S\in\mathrm{Hom}(D*\mathbf S,D'*\mathbf S)$ and, moreover, this yields a faithful functor $k\text{-}\mathsf{Systems}'\rightarrow k\text{-}\mathsf{Systems}'*\mathbf S$. By the above argument, it is also full, so $k\text{-}\mathsf{Systems}'$ and $k\text{-}\mathsf{Systems}'*\mathbf S$ are equivalent categories.
\end{proof}

In particular, Lemma~\ref{lem:sip_product_and_categories} says that the \v s\'ip product by appropriate indicators preserves (up to isomorphism) the endomorphism monoid of individual graphs, as well as the property of mutual rigidity. This way, we can reduce our problems to the context of digraphs (or even of binary relational systems), where they should be easier to tackle.

\subsection{A Cartesian product variant}\label{sect:cartesian}

Let $G$ and $H$ be two graphs. The \emph{Cartesian product} of $G$ and $H$, denoted by $G\mathbin{\square}H$, is the graph with vertex set $V(G)\times V(H)$ where two vertices $(x,u)$ and $(y,v)$ are adjacent if and only if $x=y$ and $\{u,v\}\in E(H)$, or $\{x,y\}\in E(G)$ and $u=v$. 
Sabidussi~\cite{S57} shows that the Cartesian product is useful for the construction of regular graphs (and other kinds of graphs) with a prescribed automorphism group. Here we propose a similar method to approach our problem. Let $\mathbf G=(G_1,G_2)$ be a pair of graphs with the same vertex set, $H$ another graph, and $f:V(H)\rightarrow\{1,2\}$. The \emph{Cartesian product} of $\mathbf G$ and $H$ with respect to $f$, denoted by $\mathbf G\mathbin{\underset{f}{\square}}H$, is the graph with vertex set $V(G_1)\times V(H)$ where two vertices $(x,u)$ and $(y,v)$ are adjacent if and only if $x=y$ and $\{u,v\}\in E(H)$, or $\{x,y\}\in E(G_{f(u)})$ and $u=v$.

Let $\pi_{\mathbf G}:V(G_1)\times V(H)\rightarrow V(G_1)$ and $\pi_H:V(G_1)\times V(H)\rightarrow V(H)$ denote the projections $(x,u)\mapsto x$ and $(x,u)\mapsto u$, respectively; and for any $y\in V(G_1)$ and $v\in V(H)$ let $\iota_y:V(H)\rightarrow V(G_1)\times V(H)$ and $\iota_v:V(G_1)\rightarrow V(G_1)\times V(H)$ denote the injections $u\mapsto(y,u)$ and $x\mapsto(x,v)$.

\begin{lemma}\label{lem:deficient_Cartesian_product} Let $G_1$, $G_2$ and $H$ be graphs such that $V(G_1)=V(G_2)$. Set $\mathbf G=(G_1,G_2)$ and let $f:V(H)\rightarrow\{1,2\}$ be a non-constant mapping, and $I$ a supergraph of $P=\mathbf G\mathbin{\underset{f}{\square}} H$ with vertex set $V(G_1)\times V(H)$. Assume that
\begin{enumerate}[label=(\roman*)]
   \item for every cycle $C$ of $I$ attaining the odd girth of $I$ it holds $|\pi_H(V(C))|=1$;
   \item $G_1$ and $G_2$ have the same odd girth, and for each $i\in\{1,2\}$ the union of cycles attaining the odd girth of $G_i$ is a spanning connected subgraph of $G_i$;
   \item for every edge $e\in E(I)\backslash E(P)$ it holds $|\pi_H(e)|=2$ and $\pi_H(e)\notin E(H)$;
   \item $\mathrm{Hom}(G_1,G_2)=\mathrm{Hom}(G_2,G_1)=\emptyset$;
   \item for every $\varphi,\psi\in\mathrm{End}(G_1)\cup\mathrm{End}(G_2)$, the following holds: if for every $x\in V(G_1)$ $(\varphi(x),u)$ and $(\psi(x),v)$ are adjacent vertices of $I$, then $u$ and $v$ are adjacent in $H$;
   \item for every $\varphi_H\in\mathrm{End}(H)\backslash\{\mathrm{id}_H\}$, if there is some $\varphi\in\mathrm{End}(I)$ such that $\varphi(\pi_H^{-1}(u))\subseteq\pi_H^{-1}(\varphi_H(u))$ for every $u\in V(H)$, then there exists some $v\in V(H)$ such that $f(v)\neq f(\varphi_H(v))$;
   \item $H$ is connected.
\end{enumerate}
Then, there is a monoid monomorphism $\mathrm{End}(I)\rightarrow\mathrm{End}(G_1)\cap\mathrm{End}(G_2)$.
\end{lemma}
\begin{proof} Let $\varphi\in\mathrm{End}(I)$. By (i), we have that $|\pi_H(\varphi(V(C)))|=1$ for every cycle $C$ of $I$ of minimum odd girth. By (ii), for every $u\in V(H)$ there exists a $v\in V(H)$ such that $\varphi(\pi_H^{-1}(u))\subseteq\pi_H^{-1}(v)$. This defines a mapping $\varphi_H:V(H)\rightarrow V(H)$. 
By (iii), for every $v\in V(H)$ the restriction of $\pi_{\mathbf G}$ to $\pi_H^{-1}(v)$ is in $\mathrm{Hom}(I[\pi_H^{-1}(v)],G_{f(v)}$). Therefore, the composition $$V(G_1)\overset{\iota_v}{\longrightarrow}\pi_H^{-1}(v)\overset{\varphi}{\longrightarrow}\pi_H^{-1}(\varphi_H(v))\overset{\pi_{\mathbf G}}{\longrightarrow} V(G_1)$$ is in $\mathrm{Hom}(G_{f(v)},G_{f(\varphi_H(v))})$ for every $v\in V(H)$. In particular, $f(\varphi_H(v))=f(v)$ for every $v\in V(H)$, by (iv). Since for every $x\in V(G_1)$ and every $\{u,v\}\in E(H)$ $\varphi((x,u))=(\pi_{\mathbf G}\circ\varphi\circ\iota_u(x),\varphi_H(u))$ and $\varphi((x,v))=(\pi_{\mathbf G}\circ\varphi\circ\iota_v(x),\varphi_H(v))$ are neighbours in $I$, (v) yields that $\varphi_H\in\mathrm{End}(H)$. Thus, $\varphi_H=\mathrm{id}_H$ by (vi).
Now, by (iii) and (vii), for any $u,v\in V(H)$ $\pi_{\mathbf G}\circ\varphi\circ\iota_u=\pi_{\mathbf G}\circ\varphi\circ\iota_v$. Since $f$ is non-constant, this yields a monomorphism $$\begin{matrix}\mathrm{End}(I)&\!\!\longrightarrow\!\!&\mathrm{End}(G_1)\cap\mathrm{End}(G_2) \\ \varphi&\!\!\longmapsto\!\!&\pi_{\mathbf G}\circ\varphi\circ\iota_u.\end{matrix}$$
\end{proof}

\begin{remark}
    While we believe it to be natural, we have not found this variant of the Cartesian product or its generalization to $\mathbf G=(G_1,\ldots, G_n)$ graphs on the same vertex set and $H$ with $|V(H)|=n$ in the literature, e.g.~\cite{HIK11}.
\end{remark}

\subsection{Tiling factors}\label{sec:tiling_factors}

This section is devoted to the construction of rigid $3$-indicators of arbitrarily large odd girth. The construction is based on plane graphs; see~\cite{MoharThomassen2001} for background.

A graph is \emph{$k$-connected} if it has at least $k+1$ vertices and remains connected after the removal of at most $k-1$ vertices. In a $2$-connected plane graph $G$, the boundary of each face is a cycle of $G$~\cite[Theorem 2.2.3]{MoharThomassen2001}. Such cycles are said to be \emph{facial}. We denote by $\partial G$ the facial cycle of the unbounded face, and partition the vertices of $G$ into two classes: the \emph{border} vertices, which are those in $V(\partial G)$, and the \emph{interior} vertices.

Let $g\geq 7$ be an integer. A \emph{$g$-tiling graph} is a $2$-connected plane graph $G$ satisfying:
\begin{enumerate}[label=(\roman*)]
   \item the boundary of each bounded face is a $g$-cycle;
   \item every border vertex has degree $2$ or $3$;
   \item interior vertices have degree $3$.
\end{enumerate}

\begin{figure}[h]
\centering
\includegraphics[scale=1]{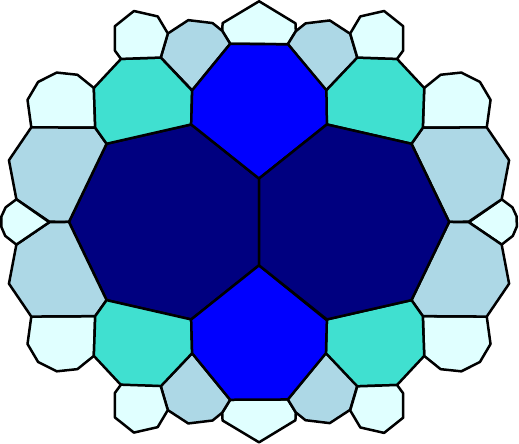}
\caption{A $7$-tiling graph.
}
\label{fig:tiling}
\end{figure}

\begin{lemma}\label{lem:tiling_girth} Let $g\geq 7$ be an integer and $G$ a $g$-tiling graph.
\begin{enumerate}[label=(\roman*)]
   \item\label{enum:tiling_graph_basic} Every edge of $G$ is in the facial cycle of exactly $2$ faces, and every vertex $v\in V(G)$ is in the facial cycle of exactly $\deg(v)$ faces.
   \item\label{enum:tiling_graph_girth} $G$ has girth $g$, and its $g$-cycles are precisely the facial cycles of its bounded faces.
\end{enumerate}
\end{lemma}
\begin{proof} 
(i) This holds in general for $2$-connected plane graphs. For the first part, see the proof of~\cite[Theorem 2.2.3]{MoharThomassen2001}. For the second part, let $E_v$ be the set of edges incident with $v$ and $F_v$ the set of faces that have $v$ in their boundary. It follows that, for every edge $e\in E_v$, the two faces having $e$ in their boundary are in $F_v$. On the other hand, the facial cycle of any face of $F_v$ goes through $v$ and exactly two edges incident with $v$. Thus, $|F_v|=|E_v|=\deg(v)$ by double counting.


(ii) Let $C$ be a cycle of $G$ of minimum length (in particular, at most $g$), and let $H$ be the subgraph of $G$ consisting of $C$ and all vertices and edges in the region bounded by $C$. We note that $H$ is a 
$g$-tiling graph with $\partial H=C$. Let $f$, $e$, $n_2$ and $n_3$ be the number of faces, edges, vertices of degree $2$ and vertices of degree $3$ of $H$, respectively. Using the definition of $g$-tiling graph (and the fact that an edge is in the boundary of exactly two faces), we see that $2n_2+3n_3=2e=g(f-1)+|C|$. By Euler's formula,
\begin{align*}
6f-g(f-1)-|C|+2n_2 &= 12 \\
(6-g)f &= 12-g+|C|-2n_n \geq 12-2g,
\end{align*}
where, in the last inequality, we used that $n_2\leq |C|\leq g$. Since $g\geq 7$, we conclude that $f\leq 2$. This implies that $H$ has exactly two faces, $H=C$, and $C$ is the facial cycle of a bounded face of $G$. In particular, $|C|=g$ by definition of $g$-tiling graph.
\end{proof}

\begin{lemma}\label{lem:tiling_factor_girth} Let $G$ be a $g$-tiling graph such that
\begin{enumerate}[label=(\roman*)]
   \item $g\geq 7$ is odd;
   \item $\partial G$ is an even cycle;
   \item there is an automorphism $x\mapsto -x$ of $G$ mapping every vertex of $\partial G$ to its antipode in $\partial G$;
   \item $\mathrm{dist}_G(x,-x)\geq g$ for every vertex $x$ of $\partial G$;
   \item there is an odd integer $h\geq (g+1)/2$ and a set $U=\{u_0,\ldots,u_{2h-1}\}$ of border vertices of degree $2$, such that $u_{i+h}=-u_i$ for $0\leq i\leq h-1$ and $\mathrm{dist}_G(u_i,u_j)\geq g$ for $0\leq i<j\leq 2h-1$.
\end{enumerate}
Consider the graphs 
\begin{itemize}
    \item $T\coloneq G\cup\{\{x,-x\}\mid\deg_G(x)=2,\ x\notin U\}\cup\{\{u_{2i},u_{2i+1}\}\mid 1\leq i\leq h-1\}$,
    \item $T'\coloneq G\cup\{\{x,-x\}\mid\deg_G(x)=2\}$, and
    \item $\overline T\coloneq G\cup\{\{x,-x\}\mid\deg_G(x)=2\}\cup\{\{u_{2i},u_{2i+1}\}\mid 0\leq i\leq h-1\}$ 
\end{itemize}
(see Figure~\ref{fig:factor}). Then, every shortest odd cycle of $\overline T$ is  in $G$.
\end{lemma}
\begin{proof} Let us divide the edges in $E(\overline T)\backslash E(G)$ into two types: the ones of the form $\{x,-x\}$, called \emph{straight}, and the others, called \emph{twisted}. We will use the following claims.

\begin{claim}\label{claim:straight_paths} \textit{Let $u,u'\in U$. If $\mathrm{dist}_{T'}(u,u')<g$, then $u'\in\{u,-u\}$.}
\end{claim}
\noindent\textit{Proof of Claim~\ref{claim:straight_paths}.}
Assume that $u'\neq u$, and let $P$ be a shortest path of $T'$ between $u$ and $u'$ of length less than $g$. By (v), $P$ uses at least one straight edge. Suppose that it uses at least two. We write $P$ as a concatenation $Q\{x,-x\}R\{y,-y\}S$, where $Q$, $R$ and $S$ are paths of $T'$ from $u$ to $x$, from $-x$ to $y$, and from $-y$ to $u'$, respectively. Then, $Q(-R)S$ is a walk of $T'$ from $u$ to $u'$. But this contradicts that $P$ is shortest. Hence, we can write $P$ as a concatenation $Q\{x,-x\}R$, where $Q$ and $R$ are paths of $G$ from $u$ to $x$ and from $-x$ to $u'$, respectively. Then, $(-Q)R$ is a walk of $G$ from $-u$ to $u'$. By (v), we must have $u'=-u$. \qedblack

\begin{claim}\label{claim:U_spread} \textit{Let $\{u,u'\}$ be a twisted edge of $\overline T$ and $P$ a path in $\overline T$ between $u$ and $u'$, other than $\{u,u'\}$ itself. Then, $\mathrm{length}(P)\geq g$.}
\end{claim}
\noindent\textit{Proof of Claim~\ref{claim:U_spread}.} Let $0\leq j_0,j'_0\leq 2h-1$ such that $u=u_{j_0}$ and $u'=u_{j'_0}$. Without loss of generality, we can assume that $j_0$ is even and $j'_0=j_0+1$. Let $\{u_{j_1},u_{j'_1}\},\ldots,\{u_{j_s},u_{j'_s}\}$ be the twisted edges of $P$ (if it has any), so that $P$ can be written as a concatenation $P_0\{u_{j_1},u_{j'_1}\}P_1\ldots\{u_{j_s},u_{j'_s}\}P_s$, where, for each $0\leq i\leq s$, $P_i$ is a path of $T'$ from $u_{j'_i}$ to $u_{j_{i+1}}$ (here, $u_{j_{s+1}}$ denotes $u_{j_0}$). We note that none of the twisted edges of $P$ is $\{u,u'\}$, because otherwise $P$ would be $\{u,u'\}$ itself. Therefore, since the twisted edges form a matching of $\overline T$, each of $P_0,\ldots,P_s$ has at least one edge.

Suppose that $\mathrm{length}(P)<g$. Then, by Claim~\ref{claim:straight_paths}, $u_{j_{i+1}}=-u_{j'_i}=u_{j'_i+h}$ for every $0\leq i\leq s$, where the addition $j'_i+h$ is modulo $2h$. Since $h$ is odd, $j_{i+1}$ and $j'_i$ have different parity, so $j_0,j_1,\ldots,j_s$ are even. This implies that $j_0=j_{s+1}\equiv j'_s+h = j_s +h+1\equiv\ldots\equiv j_0+(s+1)(h+1) \mod 2h$. Therefore, $s+1$ is divisible by $h$, and in fact $s+1=h$, because $\overline T$ has only $h$ twisted edges. Hence, $\mathrm{length}(P)\geq 2s+1=2h-1\geq g$.\qedblack

We proceed with the proof of the Lemma. By Lemma~\ref{lem:tiling_girth}, $G$ has odd girth $g$. Let $C$ be a shortest odd cycle of $\overline T$. By Claim~\ref{claim:U_spread}, $C$ has no twisted edges, i.e.~it is a cycle of $T'$. Suppose that $C$ has two straight edges. Then we can write it as a concatenation $Q\{x,-x\}R\{y,-y\}$, where $Q$ and $R$ are paths of $T'$ from $-y$ to $x$, and from $-x$ to $y$, respectively. But then $(-Q)R$ is a closed odd walk of $T'$, contradicting the fact that $C$ is shortest. If $C$ has only one straight edge $\{x,-x\}$, then $\mathrm{length}(C)\geq\mathrm{dist}_G(x,-x)+1\geq g+1$, by (iv). Hence, $C$ has no straight edges either, meaning that it is a cycle of $G$.
\end{proof}

\begin{figure}[h]
\centering
\includegraphics[scale=0.79]{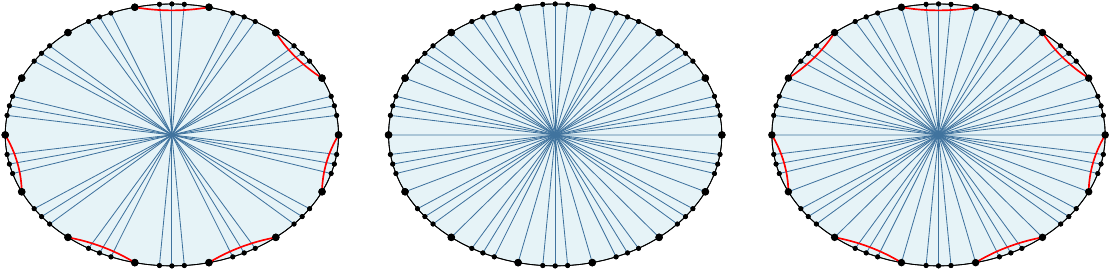}
\caption{The graphs $T$ (left), $T'$ (centre), and $\overline T$ (right) from Lemma~\ref{lem:tiling_factor_girth}. In the illustration we assume that the symmetries of the ellipse correspond to the symmetries of $G$. Only the vertices from $U$ (larger) and some other vertices of the border (smaller) have been depicted. Twisted edges appear in red, and straight edges in dark blue.}
\label{fig:factor}
\end{figure}

Let $G$ be a $g$-tiling graph satisfying (i)--(v) from Lemma~\ref{lem:tiling_factor_girth} for some set of vertices $U=\{u_0,u_1,\ldots u_{2h-1}\}\subseteq V(G)$. Assume that, additionally,
\begin{enumerate}[label=(\roman*)]
\setcounter{enumi}{5}
      \item the unique automorphism of $G$ that maps $\{u_0,u_1\}$ to itself is the identity;
   \item every $g$-cycle of $G$ has an interior vertex of $G$;
   \item the subgraph of $G$ induced by the interior vertices is connected.
\end{enumerate}
In this situation, a \emph{$g$-tiling factor} from $G$ is a pair of graphs $\mathbf T=(T,T')$, where $T$ and $T'$ are as in Lemma~\ref{lem:tiling_factor_girth}. The term refers to the role that $\mathbf T$ will play as a factor in several Cartesian products (see Section~\ref{sect:cartesian}). We collect some observations that will be useful for later. They follow from the definitions of $T$, $T'$ and $g$-tiling graph, and from Lemmas~\ref{lem:tiling_girth} and~\ref{lem:tiling_factor_girth}.

\begin{remark}\label{rem:easy_properties_tiling_factors} $T\cup\{\{u_0,u_1\}\}$ and $T'$ are cubic. $G$, $T$ and $T'$ are $2$-connected and have odd girth $g$. Every vertex of $G$, $T$ and $T'$ belongs to a $g$-cycle, and moreover the union of all the $g$-cycles of either $G$, $T$ or $T'$ is $G$ itself. 
\end{remark}

Let $G$ be any graph and $H$ a subgraph of $G$. A homomorphism $\varphi:G\rightarrow H$ is called a \emph{retraction} if the restriction of $\varphi$ to $V(H)$ is the identity, and in this case $H$ is a \emph{retract} of $G$. If $G$ has no retract other than itself, $G$ is called a \emph{core}. Cores are precisely those graphs whose endomorphisms are all automorphisms~\cite{HN92}.

\begin{lemma}\label{lem:deficient_tiling_factors_are_rigid} Let $\mathbf T=(T,T')$ be a $g$-tiling factor from a $g$-tiling graph $G$. Then
\begin{enumerate}[label=(\roman*)]
   \item\label{enum:deficient_tiling_factors_are_rigid_rigidity} $T$ is rigid;
   \item\label{enum:deficient_tiling_factors_are_rigid_automorphisms} $\mathrm{End}(T')\subseteq\mathrm{Aut}(G)$;
   \item\label{enum:deficient_tiling_factors_are_rigid_no_homomorphisms} $\mathrm{Hom}(T,T')=\mathrm{Hom}(T',T)=\emptyset$.
\end{enumerate}
\end{lemma}
\begin{proof} Let $G$, $U=\{u_0,\ldots,u_{2h-1}\}$, $T$ and $T'$ be as in the definition of $g$-tiling factor above. We shall see first that $G$ is a core. Let $H$ be a retract of $G$ and $\varphi:G\rightarrow H$ a retraction. Given an interior vertex $x$, let $x_1,x_2,x_3$ be its neighbours. By Lemma~\ref{lem:tiling_girth}, for $1\leq i<j\leq 3$, there is a unique (minimum odd girth) $g$-cycle $C_{i,j}$ through $x_i,x,x_j$. If $\varphi(C_{1,2})=\varphi(C_{2,3})$ then $\varphi(x_1)=\varphi(x_3)$, contradicting that $\varphi(C_{1,3})$ is a $g$-cycle, so $\varphi(C_{1,2}),\varphi(C_{2,3}),\varphi(C_{1,3})$ are distinct $g$-cycles of $H$. Therefore, whenever $x$ is a vertex of the retract $H$, $\varphi$ is the identity on $C_{1,2},C_{2,3},C_{1,3}$, and in particular $x_1,x_2,x_3$ are also in the retract. Since by Lemma~\ref{lem:tiling_girth} border vertices belong to less than three $g$-cycles, $H$ has an interior vertex. Hence, by (viii) every interior vertex $x$ is in $H$. Since the three $g$-cycles through $x$ are in $H$ as well, by (vii) $H$ is the whole $G$. In other words, $G$ is a core.

Now, let $T_1,T_2\in\{T,T'\}$ and $\psi\in\mathrm{Hom}(T_1,T_2)$. By Lemma~\ref{lem:tiling_girth}, every edge $e$ of $G$ is in a shortest odd cycle, so by Lemma~\ref{lem:tiling_factor_girth} $\psi(e)$ is an edge of $G$. This implies that $\psi\in\End(G)=\Aut(G)$. Since $T$ has less edges than $T'$, the case $T_1=T'$ and $T_2=T$ can be discarded, and we may assume that $T_1=T$. If $T_2=T'$, then $\{\psi(u_2),\psi(u_3)\}$ is an edge of $T'$, but not of $G$, because $\mathrm{dist}_G(\psi(u_2),\psi(u_3))\geq g$.  Hence $\psi(u_2)=-\psi(u_3)$. By Lemma~\ref{lem:tiling_girth}, the edges of $\partial G$ are in only one $g$-cycle, so $\psi\in\Aut(\partial G)$, and this leads to $u_2=-u_3$, a contradiction to $h\geq (g+1)/2$. Finally, if $T_2=T$, $\psi\in\Aut(T)$, which implies that the set $\{u_0,u_1\}$ of vertices of degree $2$ is mapped to itself. By (vi), $\psi$ is the identity.
\end{proof}

In what follows we give concrete examples of $g$-tiling factors. Let $g\geq 7$ be an odd integer, and let $G(g,1)$ be a plane graph consisting of two facial $g$-gons sharing an edge. Given an integer $i\geq 1$, we assume that $G(g,i)$ is a $g$-tiling graph and construct from it a new plane graph $G(g,i+1)$ as follows. 
\begin{itemize}
   \item[] \textit{Step 1.} For every degree $2$ vertex $x$ of $G(g,i)$ which is adjacent to a degree $3$ vertex, add a new vertex $l(x)$ in the unbounded face and join it to $x$ without crossings.
   \item[] \textit{Step 2.} Join two of these new vertices $l(x),l(y)$ whenever $x,y$ are non-adjacent but consecutive degree $2$ vertices in $\partial G(g,i)$ (i.e., there is a path $P_{x,y}$ in $\partial G(g,i)$ of length at least $2$ connecting $x$ and $y$ and using no more vertices of degree $2$); do it without crossings.
   \item[] \textit{Step 3.} Subdivide each edge $\{l(x),l(y)\}$ enough so that the bounded face it forms has $g$ vertices (that is, $g-\mathrm{length}(P_{x,y})-3\geq g-6$ times; see Lemma~\ref{lem:tiling_graph_well-defined}). 
\end{itemize}
See Figure~\ref{fig:tiling} for a depiction of $G(7,5)$. We should be wary that $G(g,i)$ is \emph{well-defined} in the sense that these instructions can be followed without creating any crossings, and that in Step~3 the bounded face formed by $\{l(x),l(y)\}$ has not already too many vertices. This is the purpose of the following:

\begin{lemma}\label{lem:tiling_graph_well-defined} For every odd $g\geq 7$ and every integer $i\geq 1$, $G(g,i)$ is a well-defined $g$-tiling graph.
\end{lemma}
\begin{proof} Clearly, this is true for $G(g,1)$. We assume that there are no problems with the construction of $G(g,j)$ for $1\leq j\leq i$, and prove the claim for $G(g,i+1)$. 

Let us first quickly note that, since $\partial G(g,i)$ is a cycle, it makes sense to ask whether two degree $2$ vertices are consecutive in $\partial G(g,i)$. By the Schoenflies theorem (\hspace{1sp}\cite[Section 2.2]{MoharThomassen2001}, or \cite[Chapter III]{Bing1983}), the steps above can be performed without creating any crossings. For example, one can assume that the point set of $\partial G(g,i)$ is the circumference of radius one centered at the origin, that the new edges of the form $\{x,l(x)\}$ are segments of length $s>0$ of lines through the origin, and that the rest of new edges are segments of the circumference of radius $1+s$ centered at the origin. Moreover, the following claim certifies that in the construction of $G(g,i+1)$ there is no issue with Step 3, because it implies that at least $g-6$ subdivisions are always required. (It is stated in a strong form that will be useful for later.) 

\begin{claim}\label{claim:two_deg2_neighbours} \textit{Every border vertex of $G(g,i)$ of degree 3 is adjacent to a degree $2$ vertex which also has a degree $2$ neighbour. In particular, there is an injection from the set of border vertices of $G(g,i)$ of degree $3$ to the set of border vertices of $G(g,i+1)$ of degree $3$ which are already in $G(g,i)$. (See Figure~\ref{fig:three_extensions} for an illustration.)}
\end{claim}
\noindent\textit{Proof of Claim~\ref{claim:two_deg2_neighbours}.}
The claim holds for $i=1$, so assume that $i\geq 2$. Let $x$ be a border vertex of $G(g,i)$ of degree $3$. We distinguish two cases.
\begin{itemize}
   \item[] \textit{Case 1: $x$ is a vertex of $G(g,i-1)$.} In this case, $x$ has degree $2$ in $G(g,i-1)$, for otherwise it would not be in the border of $G(g,i)$. Moreover, $x$ has a neighbour $w$ of degree $3$ in $G(g,i-1)$, and its other neighbour $y$ has degree $2$ in $G(g,i-1)$, again because $x$ is in the border of $G(g,i)$. Therefore, $l(x)$ has degree $2$ in $G(g,i)$, and has a neighbour of degree $2$ because of the subdivision step (by induction, during the construction of $G(g,i)$, at least $g-6$ subdivisions are required each time).
   \item[] \textit{Case 2: $x$ is not a vertex of $G(g,i-1)$.} In this case, $x=l(y)$ for some degree $2$ vertex $y$ of $G(g,i-1)$, which has two degree $3$ neighbours $w,z$ in $G(g,i-1)$. By induction, $z$ has a degree $2$ neighbour $t\neq y$. Therefore, the provisional edge $\{x,l(t)\}$ from Step 2 has to be subdivided $g-5\geq 2$ times. \qedblack
\end{itemize}

\noindent Therefore, the steps in the definition of $G(g,i+1)$ make sense. Finally, since by assumption $G(g,i)$ is a $g$-tiling graph, we see that $G(g,i+1)$ is a $g$-tiling graph too. 
\end{proof}

\begin{figure}[h]
\centering
\includegraphics[scale=0.9]{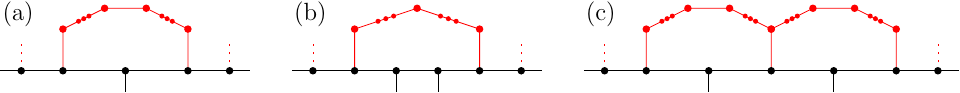}
\caption{As a consequence of Claim~\ref{claim:two_deg2_neighbours}, for any degree $3$ border vertex $x$ of $G(g,i)$, the closed ball of $\partial G(g,i)$ of radius $2$ $B_{\partial G(g,i)}(x,2)=\{y\in V(\partial G(g,i))\mid\mathrm{dist}_{\partial G(g,i)}(x,y)\leq 2\}$ has either (a) no other vertex of degree $3$ (in $G(g,i)$), (b) exactly one other vertex of degree $3$, which is adjacent to $x$, or (c) exactly one other vertex of degree $3$, which is at distance $2$ from $x$ (in $\partial G(g,i)$). Thus, the number of new vertices of $G(g,i+1)$ in the new bounded face of $x$, and their degrees, are completely determined in each of these three cases. In black, the vertices and edges of $G(g,i)$, with the unbounded face at the top. In red, the new vertices and edges of $G(g,i+1)$. The smaller dots represent the appropriate number of extra subdivisions needed. The red dotted lines represent possible other new edges of $G(g,i+1)$.}
\label{fig:three_extensions}
\end{figure}
A vertex of $G(g,i)$ is \emph{old} if it is also a vertex of $G(g,i-1)$, otherwise it is \emph{new} (all vertices of $G(g,1)$ are new). 

\begin{lemma}\label{lem:aut=Z22} For every odd $g\geq 7$ and every integer $i\geq 1$, $\mathrm{Aut}(G(g,i))\cong\mathbb Z_2^2$. More precisely, let $V_3$ be the set of the two vertices of $G(g,1)$ of degree $3$, and $V_3^{\perp}$ the set of the two vertices of $G(g,1)$ at distance $(g-1)/2$ from $V_3$. Then
\begin{enumerate}[label=(\roman*)]
    \item \[\begin{matrix}
    \mathrm{Aut}(G(g,1))&\!\!\longrightarrow\!\!&\langle\sigma\rangle\times\langle\tau\rangle \\
    \varphi&\!\!\longmapsto\!\!&(\varphi\vert_{V_3},\varphi\vert_{V_3^{\perp}})
    \end{matrix}\]
    is a group isomorphism, where $\sigma$ and $\tau$ are the transpositions on $V_3$ and $V_3^{\perp}$, respectively;
    \item for every integer $i\geq 2$,\[\begin{matrix}
    \mathrm{Aut}(G(g,i))&\!\!\longrightarrow\!\!&\mathrm{Aut}(G(g,i-1)) \\
    \varphi&\!\!\longmapsto\!\!&\varphi\vert_{V(G(g,i-1))}
    \end{matrix}\]
    is a group isomorphism.
\end{enumerate}
\end{lemma}
\begin{proof} If $i=1$, any $\varphi\in\mathrm{Aut}(G(g,i))$ induces a permutation of $V_3$ and one of $V_3^{\perp}$. Moreover, given a permutation $\sigma'$ of $V_3$ and a permutation $\tau'$ of $V_3^{\perp}$, there is a unique $\varphi\in\mathrm{Aut}(G(g,i))$ with $\varphi\vert_{V_3}=\sigma'$ and $\varphi\vert_{V_3^{\perp}}=\tau'$. Therefore, the mapping $\varphi\mapsto(\varphi\vert_{V_3},\varphi\vert_{V_3^{\perp}})$ given above is a group isomorphism. For $i\geq 2$, we use the following claims.

\begin{claim}\label{claim:disjoint_paths} \textit{For $i\geq 3$ (resp.~for $i=2$), the subgraph of $\partial G(g,i)$ induced by the new vertices of $G(g,i)$ of degree $2$ is a disjoint union of paths of lengths $g-4$ or $g-5$ (resp.~of length $g-4$), and the subgraph of $\partial G(g,i)$ induced by the old vertices of $G(g,i)$ of degree $2$ is a disjoint union of paths of length less than $g-5$ (resp.~of length $g-5$).}
\end{claim}
\noindent\textit{Proof of Claim~\ref{claim:disjoint_paths}.}
The case $i=2$ is immediate. For $i\geq 3$, it follows from Figure~\ref{fig:three_extensions}. \qedblack

\begin{claim}\label{claim:restrictions} \textit{For $i\geq 2$, every $\varphi\in\mathrm{Aut}(G(g,i))$ sends $V(\partial G(g,i))$ and $V(G(g,i-1))$ to themselves. Consequently, the restriction of $\varphi$ to $V(\partial G(g,i))$ (resp.~$V(G(g,i-1))$) is an automorphism of $\partial G(g,i)$ (resp.~$G(g,i-1)$).}
\end{claim}
\noindent\textit{Proof of Claim~\ref{claim:restrictions}.}
By Claim~\ref{claim:two_deg2_neighbours}, $V(\partial G(g,i))$ is precisely the set of vertices of $G(g,i)$ either of degree $2$ or with a neighbour of degree $2$. Hence, $\varphi$ sends $V(\partial G(g,i))$ to itself. From Claim~\ref{claim:disjoint_paths} we see that $\varphi$ sends new vertices of $G(g,i)$ of degree $2$ to new vertices of $G(g,i)$ of degree $2$. This implies that, in fact, $\varphi$ sends new vertices (always of $G(g,i)$) to new vertices, so it sends old vertices to old vertices. \qedblack


By Claim~\ref{claim:restrictions}, it is enough to see that every $\psi\in\Aut(G(g,i-1))$ extends to a unique $\varphi\in\Aut(G(g,i))$. Let $\{x,y\}$ be an edge of $G(g,i)$, where $x$ is a new vertex of $G(g,i)$ and $y$ is old. Then $x=l(y)$. Also, $y$ has degree $2$ in $G(g,i-1)$, and is adjacent to a vertex of degree $3$ in $G(g,i-1)$. The same is true for $\psi(y)$; therefore, $l(\psi(y))$ is a new vertex of $G(g,i)$. Since $l(\psi(y))$ is the only new neighbour of $\psi(y)$, any $\varphi\in\Aut(G(g,i))$ extending $\psi$ must satisfy $\varphi(l(y))=l(\psi(y))$.

Let $x,x'$ be two new vertices of $G(g,i)$ of the form $l(y),l(y')$, respectively. Assume that $y,y'$ are non-adjacent but consecutive degree $2$ vertices in $\partial G(g,i-1)$. We consider the path $P$ from $x$ to $x'$ using only new vertices. Any $\varphi\in\Aut(G(g,i))$ extending $\psi$ must send $P$ to a path between $l(\psi(y))$ and $l(\psi(y'))$ using only new vertices. There is exactly one such path $Q$. Moreover, $P$ and $Q$ have the same length. Hence, such a $\varphi\in\Aut(G(g,i))$ exists, and is completely determined.
\end{proof}

A subgraph $H$ of a graph $G$ is \emph{convex} if, for any two vertices of $H$, all shortest paths between them in $G$ are also in $H$.

\begin{lemma}\label{lem:existence_tiling_factors} For every odd $g\geq 7$ there is a positive integer $i_g$ such, that for every $i\geq i_g$, we can find $U\subseteq V(G(g,i))$ such that the $g$-tiling graph $G(g,i)$ together with $U$ satisfies the conditions (i)--(viii) from the definition of $g$-tiling factor.
\end{lemma}
\begin{proof} (ii)--(iii) Let $\varphi$ be the automorphism of $G(g,i)$ which permutes the two bounded faces and the two degree~$3$ vertices of $G(g,1)$, obtained with Lemma~\ref{lem:aut=Z22}. By Claim~\ref{claim:restrictions}, its restriction to the border vertices is an automorphism of the facial cycle $\partial G(g,i)$. If $i=1$, $|V(\partial G(g,i))|=2g-2$, and $\varphi$ sends every vertex of $\partial G(g,i)$ to its antipode. Assume that $i\geq 2$. It follows from Claim~\ref{claim:two_deg2_neighbours} that $\partial G(g,i)$ has at least three vertices of $\partial G(g,i-1)$, say $x,y,z$. By induction, the restriction of $\varphi$ to $\partial G(g,i-1)$ is the order $2$ rotation, so $\varphi$ preserves the circular ordering of $x,y,z$, which clearly is the same in $\partial G(g,i-1)$ as in $\partial G(g,i)$. Therefore $\varphi$ is an order $2$ rotation of $\partial G(g,i)$. In particular, the cycle $\partial G(g,i)$ is even.

(iv) Given $x\in V(G(g,i))$, we denote by $-x$ the image of $x$ by the automorphism $\varphi$ from the previous paragraph. We note that, as long as $x$ is indeed a vertex of $G(g,i)$, by Lemma~\ref{lem:aut=Z22} $-x$ is independent of $i$. Let us fix $g$. We will prove that $f(i)\coloneq\min \{\mathrm{dist}_{G(g,i)}(x,-x)\mid x\in V(\partial G(g,i))\}$ is an increasing unbounded function of $i$. It follows from Figure~\ref{fig:three_extensions} that $G(g,i-1)$ is a convex subgraph of $G(g,i)$. Hence, the distance function on $G(g,i)$ is also independent of $i$; here we denote it just by $\mathrm{dist}$. It will be enough to see that, given any new vertex $x$ of $G(g,i+1)$, $\mathrm{dist}(x,-x)>f(i)$. Indeed, since any border vertex of $G(g,i)$ is an interior vertex of $G(g,i+(g-1)/2)$, this will imply that $f(i+(g-1)/2)>f(i)$.

Let $x$ be a new vertex of $G(g,i+1)$. Claim~\ref{claim:two_deg2_neighbours} implies that $|V(\partial G(g,i+1))|\geq|V(\partial G(g,i))|\geq\ldots\geq |V(\partial G(g,2))|=4g-10$. By Figure~\ref{fig:three_extensions}, the subgraph induced by the new vertices of $G(g,i+1)$ is a disjoint union of paths of length at most $2g-8$. 
Therefore, if we let $C$ and $D$ be the components of $x$ and $-x$ in this linear forest, then $C\neq D$. Moreover, by Claim~\ref{claim:restrictions}, $D=-C$. We distinguish two cases.
\begin{itemize}
   \item[] \textit{Case 1: $C$ is adjacent to exactly two old vertices of $G(g,i+1)$ ((a) and (b) in Figure~\ref{fig:three_extensions}).} Let $y_1$ and $y_2$ be these two old vertices; then the old vertices adjacent to $-C$ are precisely $-y_1$ and $-y_2$. Since old vertices have at most one new neighbour, $y_1,y_2,-y_1,-y_2$ are pairwise different. If one geodesic between $x$ and $-x$ goes through $y_1$ and $-y_1$, we are done. Therefore, assume that $\mathrm{dist}(x,-x)=\mathrm{dist}_{\partial G(g,i+1)}(x,y_1)+\mathrm{dist}(y_1,-y_2)+\mathrm{dist}_{\partial G(g,i+1)}(-y_2,-x)$. Since $\mathrm{dist}(y_2,-y_2)\leq 3+\mathrm{dist}(y_1,-y_2)$, we have that $\mathrm{dist}(x,-x)\geq\mathrm{dist}(y_1,-y_2)+g-3\geq\mathrm{dist}(y_2,-y_2)+g-6$.
   \item[] \textit{Case 2: $C$ is adjacent to exactly three old vertices of $G(g,i+1)$ ((c) in Figure~\ref{fig:three_extensions}).} Let $y_1,y_2,y_3$ be these three vertices. We assume that $y_2$ is the one of them adjacent to the new vertex of degree $3$ and that $y_1$ shares a bounded face with $x$. Then, $-y_1,-y_2,-y_3$ are the three old vertices adjacent to $-C$, $-y_2$ is the one of them adjacent to a new vertex of degree $3$ and $-y_1$ shares a bounded face with $-x$. Again, $y_1,y_2,y_3,-y_1,-y_2,-y_3$ are pairwise different. If $y_3$ (resp.~$-y_3$) is in a geodesic between $x$ and $-x$, then so is $y_2$ (resp.~$-y_2$). Therefore any such geodesic must go through $y_1$ or $y_2$, and through $-y_1$ or $-y_2$. Thus we are in a situation like in Case~1, and the same argument applies.
\end{itemize}

(v)--(vi) We call a pair $(u,u')$ of degree $2$ border vertices with $\mathrm{dist}_{G(g,i)} (u,u')\geq g$ and $\mathrm{dist}_{G(g,i)} (u,-u')\geq g$ a \emph{candidate pair}. Assume that $g$ is fixed; we denote by $b(i)$ the number of border vertices of $G(g,i)$. Note that $b(i)\geq 2f(i)$, where $f$ is the function from part (iv). Since $f$ is increasing and unbounded, we can assume that $b(i)$ is as large as we want. By Figure~\ref{fig:three_extensions}, $G(g,i)$ has at least $\min\{\tfrac{g-3}{g-1},\tfrac{g-4}{g-2},\tfrac{2g-8}{2g-5}\}b(i)=\tfrac{g-4}{g-2}b(i)\geq\tfrac{3}{5}b(i)$ border vertices of degree $2$; let $u$ be any of them. Note that the closed ball $B_{G(g,i)}(u,g-1)$ on $G(g,i)$ with centre $u$ and radius $g-1$ has less than $2^{g}$ elements. Therefore, the number of vertices $u'$ such that $(u,u')$ is a candidate pair is at least $\tfrac{3}{5}b(i)-|B_{G(g,i)}(u,g-1)|-|B_{G(g,i)}(-u,g-1)|\geq\tfrac{3}{5}b(i)-2^{g+1}$. We conclude that there are at least $\tfrac{9}{25}b(i)^2-\tfrac{3}{5}2^{g+1}b(i)$ candidate pairs.

Let $\varphi\in\mathrm{Aut}(G(g,i))$. Recall that by Claim~\ref{claim:restrictions} the restriction of $\varphi$ to the border vertices is an automorphism of $\partial G(g,i)$. To that we can add the following.

\begin{claim}\label{claim:border_determines} \textit{For $i\geq 1$, every $\varphi\in\mathrm{Aut}(G(g,i))$ is uniquely determined by its restriction to $V(\partial G(g,i))$.}
\end{claim}
\noindent\textit{Proof of Claim~\ref{claim:border_determines}.}
For $i=1$ it is immediate. For $i\geq 2$, given $x\in V(\partial G(g,i-1))$, we distinguish two cases.
\begin{itemize}
   \item[]\textit{Case 1: $x$ has degree $2$ in $G(g,i-1)$.} In this case, $x$ is either a vertex of $\partial G(g,i)$ or the unique degree $3$ neighbour of a vertex of $\partial G(g,i)$, so $\varphi(x)$ is determined.
   \item[]\textit{Case 2: $x$ has degree $3$ in $G(g,i-1)$.} In this case $x$ is the unique interior neighbour of a vertex of $\partial G(g,i)$, and we are done by Claim~\ref{claim:restrictions}. 
\end{itemize}
Therefore, the restriction of $\varphi$ to $V(\partial G(g,i-1))$ is uniquely determined. But, by Claim~\ref{claim:restrictions}, the restriction of $\varphi$ to $V(G(g,i-1))$ is an automorphism of $G(g,i-1)$, so by induction it is uniquely determined. \qedblack\\

\noindent Let $(u,u')$ be a candidate pair. Since $u\neq u'\neq -u$, Claim~\ref{claim:border_determines} implies that $\varphi$ is determined by the image of $(u,u')$. If $\varphi(u,u')=(u,u')$ then $\varphi$ must be the identity, so, if we see that $\varphi(u,u')\neq(u',u)$ for every $\varphi\in\mathrm{Aut}(G(g,i))$, we will have shown that the pair $(u,u')$ satisfies (vi).

Given $\varphi\in\mathrm{Aut}(G(g,i))$, we say that a candidate pair $(u,u')$ is $\varphi$-\emph{bad} if $\varphi(u,u')=(u',u)$. Note that there are at most $b(i)$ $\varphi$-bad candidate pairs. By Lemma~\ref{lem:aut=Z22}, $G(g,i)$ has only three non-trivial automorphisms, so there are at least $\tfrac{9}{25}b(i)^2-\tfrac{3}{5}2^{g+1}b(i)-3b(i)$ candidate pairs satisfying (vi). We take $u_0,u_1$ so that $(u_0,u_1)$ is one of such pairs.

Now we finish the construction of $U$ recursively. Given an incomplete sequence of $2\leq j\leq h-1$ vertices $u_0,u_1,\ldots,u_{j-1}$ satisfying (v), we let $u_{j}$ be one of the remaining $\tfrac{3}{5}b(i)-2^{g+1}j$ (at least) degree $2$ border vertices at distance at least $g$ from $\{u_0,u_1,\ldots,u_{j-1},-u_0,-u_1,\ldots,-u_{j-1}\}$ (by the proof of (iv) above, any two antipodal vertices are at distance at least $g$ as well).

(vii) It follows from the definition of $G(g,i)$ that this holds for $i\geq 2$.

(viii) Clearly this is true for $i=2$. For a general $i\geq 2$, note that, by Figure~\ref{fig:three_extensions}, every border vertex of $G(g,i)$ of degree $3$ is adjacent to an interior vertex, and every vertex of degree $2$ which is interior in $G(g,i+1)$ is adjacent to a degree $3$ vertex. Since all interior vertices of $G(g,i+1)$ are of this form, or they are already interior in $G(g,i)$, the condition is satisfied by induction.
\end{proof}

\subsection{Infinite families of mutually rigid regular graphs}\label{sec:mutually_rigid}

\rigid*
\begin{proof} 
The key ingredient is the following claim.

\begin{customclaim}{1.2.1}\label{claim:strong_induction} \textit{For every $d\geq 3$ and every odd $g\geq 7$ there exists a $d$-indicator $S(d,g)$ of odd girth $g$ satisfying the hypotheses of Lemma~\ref{lem:sip_product_and_categories} (as a $1$-tuple of indicators).}
\end{customclaim}
\noindent\textit{Proof of Claim~\ref{claim:strong_induction}.}
The proof goes by induction. The base case is when $d\in\{3,4,5\}$ (see Figure~\ref{fig:basic_cases}). We consider a $g$-tiling factor $\mathbf T=(T,T')$ obtained from a $g$-tiling graph $G$, and $\overline T$ as in Lemma~\ref{lem:tiling_factor_girth}. The existence of such graphs is guaranteed by Lemmas~\ref{lem:tiling_graph_well-defined} and~\ref{lem:existence_tiling_factors}. Let $u_0,u_1$ be the two vertices of $T$ of degree $2$. It follows from Remark~\ref{rem:easy_properties_tiling_factors}, Lemma~\ref{lem:deficient_tiling_factors_are_rigid} and Claim~\ref{claim:U_spread} that, if $d=3$, we can take $T$, together with the distinguished pair $(u_0,u_1)$, as our $d$-indicator $S(3,g)$. For the other cases, we will use the following fact about the graph product from Section~\ref{sect:cartesian}.

\begin{fact}\label{fact:products} For any graph $H$ without odd cycles of length $g$ or less, any $v,v'\in V(H)$ and any $f:V(H)\rightarrow\{1,2\}$,
\begin{enumerate}[label=(\roman*)]
\vspace{-1mm}
   \item every minimum odd cycle $C$ of $\mathbf T\mathbin{\underset{f}{\square}}H$ satisfies that $|\pi_H(V(C))|=1$;
   \vspace{-1mm}
   \item every path $Q$ in $\mathbf T\mathbin{\underset{f}{\square}}H$ from $(u_0,v)$ to $(u_1,v')$ has length at least $g$.
\end{enumerate}
\end{fact}
\vspace{-2mm}
\noindent\textit{Proof of Fact~\ref{fact:products}.} (i) Recall that $T$ and $T'$ are subgraphs of $\overline T$, and note that $|C|\leq g$. Since $H$ has no odd cycles of length $g$ or less, $H[\pi_H(V(C))]$ is bipartite. In particular, $C$ uses an even number of edges $e$ with $|\pi_H(e)|=2$. Hence, after eliminating from $\pi_{\mathbf T}(C)$ the eventual consecutive repeated vertices, we obtain a closed odd walk in $\overline T$. By Lemmas~\ref{lem:tiling_girth}\ref{enum:tiling_graph_girth} and~\ref{lem:tiling_factor_girth}, it has length $g$ and $|\pi_H(V(C))|=1$.

(ii) 
Eliminating consecutive repeated vertices from $\pi_{\mathbf T}(Q)$, we obtain a walk from $u_0$ to $u_1$ in $\overline T\backslash\{\{u_0,u_1\}\}$. By Claim~\ref{claim:U_spread}, $\mathrm{length}(Q)\geq\mathrm{dist}_{\overline T\backslash\{\{u_0,u_1\}\}}(u_0,u_1)\geq g$.
\hfill$\blacklozenge$

\begin{figure}[h]
\centering
\includegraphics[scale=0.9]{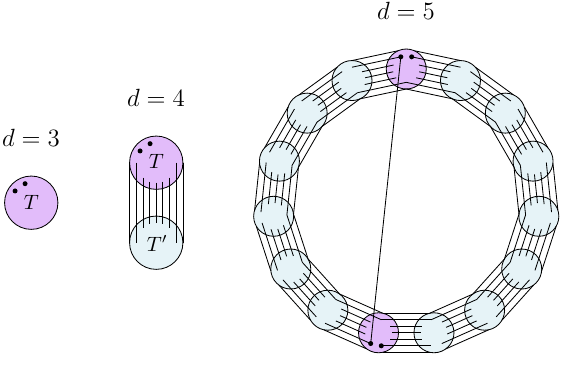}
\caption{The three basic cases for $g=7$. Subgraphs isomorphic to $T$ appear in purple, and those isomorphic to $T'$ appear in light blue. The marked vertices represent the corresponding copies of $u_0$ and $u_1$.}
\label{fig:basic_cases}
\end{figure}

For the case $d=4$, we consider the complete graph $K_2$ of order two, a vertex $v$ of it, and the mapping $f:V(K_2)\rightarrow\{1,2\}$ defined by $f^{-1}(1)=\{v\}$. By Fact~\ref{fact:products}, Remark~\ref{rem:easy_properties_tiling_factors} and Lemma~\ref{lem:deficient_tiling_factors_are_rigid}\ref{enum:deficient_tiling_factors_are_rigid_no_homomorphisms}, the graph $I=\mathbf T\mathbin{\underset{f}\square} K_2$ satisfies the hypotheses (i)--(iv), (vi) and (vii) of Lemma~\ref{lem:deficient_Cartesian_product}. Let us see that (v) holds as well. By Lemma~\ref{lem:deficient_tiling_factors_are_rigid}\ref{enum:deficient_tiling_factors_are_rigid_rigidity}, it will be enough to show the following.

\begin{fact}\label{fact:feynman} For any $\varphi,\psi\in\End(T')$, there exists an $x\in V(\overline T)$ such that $\{\varphi(x),\psi(x)\}\notin E(\overline T)$.
\end{fact}
\noindent\textit{Proof of Fact~\ref{fact:feynman}.}
By Lemma~\ref{lem:deficient_tiling_factors_are_rigid}\ref{enum:deficient_tiling_factors_are_rigid_automorphisms} and Claim~\ref{claim:restrictions}, $\varphi$ and $\psi$ are injective and send interior vertices of $G$ to interior vertices. Moreover, it follows from the definition of $g$-tiling graph that the neighbours of an interior vertex have degree $3$ in $G$, so in $\overline T$ they are incident only to edges of $G$. Let $x$ be an interior vertex of $G$, and let $x_1,x_2,x_3$ be its three neighbours. Suppose that $\{\varphi(x),\psi(x)\}, \{\varphi(x_1),\psi(x_1)\},\{\varphi(x_2),\psi(x_2)\},\{\varphi(x_3),\psi(x_3)\}\in E(\overline T)$. Without loss of generality, $\varphi(x_1)=\psi(x)$ and $\psi(x_2)\neq\varphi(x)$ (see Figure~\ref{fig:Feynman_diagram} for an illustration). Note that the sequence $\varphi(x),\varphi(x_2),\psi(x_2),\psi(x)$ defines a closed walk $W$ in $\overline T$. Moreover, all these vertices are different, so in fact $W$ is a $4$-cycle. Since by Lemma~\ref{lem:tiling_girth} $G$ has girth $g$, one of the edges of $W$ is not in $E(G)$. That is a contradiction. 
\hfill$\blacklozenge$

\begin{figure}[h]
\centering
\includegraphics[scale=1]{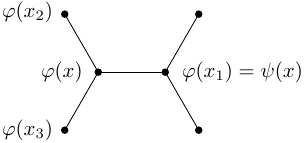}
\caption{Situation described during the proof of Fact~\ref{fact:feynman}.}
\label{fig:Feynman_diagram}
\end{figure}

\noindent Hence, by Lemma~\ref{lem:deficient_Cartesian_product}, $I$ is rigid. Let us call $\mathrm{in}\,I$ and $\mathrm{out}\,I$ the vertices of $I$ $(u_0,v)$ and $(u_1,v)$, respectively. It is clear from Remark~\ref{rem:easy_properties_tiling_factors} that, together with this distinguished pair of vertices, $I$ is a $4$-indicator. Arguing as in the case $d=3$ and using Fact~\ref{fact:products}, we see that $I$ satisfies all the properties we require $S(4,g)$ to satisfy.

In the case $d=5$, we do a similar construction with the odd cycle $C_{2g+1}$ in the place of $K_2$. Let $v_0$ and $v_1$ be two vertices of $C_{2g+1}$ at distance $g$ from each other, and let  $f:V(C_{2g+1})\rightarrow\{1,2\}$ be defined by $f^{-1}(1)=\{v_0,v_1\}$. We consider the product $P=\mathbf T\mathbin{\underset{f}\square} C_{2g+1}$, and we denote by $I$ the graph obtained from $P$ by adding the extra edge $e=\{(u_1,v_0),(u_0,v_1)\}$. We claim that $I$ satisfies the hypotheses of Lemma~\ref{lem:deficient_Cartesian_product}. We check (i), (v) and (vi); the rest are trivial or can be verified as in the case $d=4$. For (i), note that a cycle of $I$ using $e$ has length at least $g+1$, so by Remark~\ref{rem:easy_properties_tiling_factors} it is not a shortest odd cycle. The rest of the cycles of $I$ are in $P$, so we can apply Fact~\ref{fact:products}. To check (v), consider $\varphi,\psi\in\mathrm{End}(T')$ and $u,v\in V(C_{2g+1})$. If $u=v_0$ and $v=v_1$, and $(\varphi(x),u)$ and $(\psi(x),v)$ are adjacent in $I$ for every $x\in V(T')$, then $\varphi$ maps every vertex to $u_1$, a contradiction. Otherwise, 
in the case $u=v$, we apply Fact~\ref{fact:feynman}. 
Finally, let us check (vi). If $\psi$ is an endomorphism of $C_{2g+1}$, then, in fact, $\psi\in\mathrm{Aut}(C_{2g+1})$. Assume that $\psi$ maps $\{v_0,v_1\}$ to itself. Then $\psi$ permutes $v_0$ and $v_1$; otherwise, $\psi$ is the identity. Suppose that there is some $\varphi\in\mathrm{End}(I)$ such that $\varphi(\pi_{C_{2g+1}}^{-1}(u))\subseteq\pi_{C_{2g+1}}^{-1}(\psi(u))$ for every $u\in V(C_{2g+1})$. This implies that the compositions
$$V(T)\overset{\iota_{v_0}}{\longrightarrow}\pi_{C_{2g+1}}^{-1}(v_0)\overset{\varphi}{\longrightarrow}\pi_{C_{2g+1}}^{-1}(v_1)\overset{\pi_{\mathbf T}}{\longrightarrow} V(T),$$
$$V(T)\overset{\iota_{v_1}}{\longrightarrow}\pi_{C_{2g+1}}^{-1}(v_1)\overset{\varphi}{\longrightarrow}\pi_{C_{2g+1}}^{-1}(v_0)\overset{\pi_{\mathbf T}}{\longrightarrow} V(T)$$
are in $\mathrm{End}(T)$. By Lemma~\ref{lem:deficient_tiling_factors_are_rigid}, $\pi_{\mathbf T}\circ\varphi\circ\iota_{v_0}=\pi_{\mathbf T}\circ\varphi\circ\iota_{v_1}$ is the identity. But then, $\varphi$ sends $e$ to  $\varphi(e)=\varphi(\{(u_1,v_0),(u_0,v_1)\})=\{(u_1,v_1),(u_0,v_0)\}$, which is not an edge of $I$. This contradiction shows that, indeed, (vi) holds. Therefore, by Lemma~\ref{lem:deficient_Cartesian_product}, we conclude that $I$ is rigid. It is clear from Remark~\ref{rem:easy_properties_tiling_factors} that, together with the vertex pair $((u_0,v_0),(u_1,v_1))$, $I$ is a $5$-indicator. If we show that $I$ satisfies (v) of Lemma~\ref{lem:sip_product_and_categories}, the same arguments as before will be enough to see that $I$ satisfies all the properties that $S(5,g)$ has to satisfy. Let $Q$ be a path in $I$ from $(u_0,v_0)$ to $(u_1,v_1)$. If $Q$ does not use $e$, then it has length at least $g$ by Fact~\ref{fact:products}. And if $Q$ uses $e$, then it has a subpath in $P$ from $(u_0,v_0)$ to either $(u_1,v_0)$, in which case we can apply Fact~\ref{fact:products} again, or $(u_0,v_1)$, in which case it also has length at least $g$.

We now address the case $d\geq 6$. By induction, we know that there exists a $(d-3)$-indicator $S(d-3,g+2)$ of odd girth $g+2$ satisfying the hypotheses of Lemma~\ref{lem:sip_product_and_categories}. Let $D$ be a connected rigid digraph in which every vertex has total degree $d-3$ (for the existence of such digraphs, see Remark~\ref{rem:sausages}). We consider the \v s\'ip product $H=D*S(d-3,g+2)$, and we let $f:V(H)\rightarrow\{1,2\}$ be defined as $f^{-1}(1)=\{v\}$, where $v$ is an arbitrary vertex of $H$. By Lemma~\ref{lem:sip_product_and_categories} and Observation~\ref{obs:odd_girth_sip_product}, $H$ is connected, rigid, $(d-3)$-regular, and has odd girth $g+2$. The same arguments as above show that the graph $I=\mathbf T\mathbin{\underset{f}\square} H$ satisfies the hypotheses of Lemma~\ref{lem:deficient_Cartesian_product}. Hence, $I$ is rigid. By Remark~\ref{rem:easy_properties_tiling_factors}, $I$ together with the distinguished vertex pair $((u_0,v),(u_1,v))$ is a $d$-indicator. Arguing as before, one sees that $I$ satisfies all the properties that $S(d,g)$ has to satisfy.
\qedblack

To prove the theorem, it is enough to take the infinite family $\mathcal F_1(d)$ of mutually rigid digraphs of constant total degree $d\geq 3$ from Remark~\ref{rem:sausages}, and any $d$-indicator of odd girth $g\geq 7$ satisfying the hypotheses of Lemma~\ref{lem:sip_product_and_categories}, and to apply Lemma~\ref{lem:sip_product_and_categories}. By Observation~\ref{obs:odd_girth_sip_product}, the resulting family will consist of $d$-regular graphs of odd girth $g$.
\end{proof}

\section{Monoids from regular graphs}

This section contains the proof of Theorem~\ref{teo:babaipultr}. As in Section~4, the argument is divided into two steps: first, one proves a version of the theorem for binary relational systems (Lemma~\ref{lem:step4}), and then, one tries to find graphs with similar properties. The first step is now more involved; at its heart, it relies on Cayley graphs, but, unlike in the case of the representation of groups~\cite{HN73}, some extra refinement is needed in order to homogenise the indegrees. This is done in Section~\ref{sec:homogenization}. The second step, along with the proof of Theorem~\ref{teo:babaipultr}, can be found in Section~\ref{sec:babaipultr}. This time, much less effort is required, since all the work has already been done in Section~4.

\subsection{Representing monoids with binary relational systems}\label{sec:homogenization}

In the present section we construct a binary relational system in which each vertex has the same total degree, with a prescribed endomorphism monoid. This is achieved by a sequence of intermediate constructions; each of them has the desired endomorphism monoid, and the total degrees are progressively homogenised. 

In order to better describe this homogenisation process, we introduce some notions. Given a binary relational system $D$, we say that two vertices $v,w$ are in the same \emph{endomorphism class} if there are $f,g\in\mathrm{End}(D)$ with $f(v)=w$ and $g(w)=v$, and in this case we write $v\sim w$. If $f,g$ are both automorphisms, we say that $v,w$ are in the same \emph{automorphism class}, and we write $v\approx w$. We denote by $[v]$ and $\llbracket v\rrbracket$ the endomorphism and automorphism classes of $v$, respectively. We define an order relation $\preceq$ on the set $V(D)/\!\sim$ of endomorphism classes of $V(D)$ as $[v]\preceq [w]$ whenever there is some $f\in\mathrm{End}(D)$ with $f(v)=w$. This way, we obtain a poset that we denote by $\mathscr P_D$. By a \emph{connected component} of $\mathscr P_D$ we will refer to a subset of $V(D)/\!\sim$ which induces a connected component of the comparability graph of $\mathscr P_D$, i.e.~the graph on vertex set $V(D)/\!\sim$ where two elements are adjacent if and only if they are comparable in $\mathscr P_D$. We denote $V(D)/\!\approx$ by $\mathscr A_D$.

Let $D$ and $D'$ be binary relational systems. We denote by $D\leq_{\textrm{End}} D'$ that $D$ is an induced subsystem of $D'$ and $g\mapsto g\vert_{V(D)}$ defines a bijection $\Psi:\textrm{End}(D')\rightarrow\textrm{End}(D)$. 
Let us remark a couple of basic facts.
\begin{lemma}\label{lem:leqend} If $D,D'$ are two binary relational systems with $D\leq_{\End}D'$, then
\begin{enumerate}[label=(\roman*)]
   \item $\Psi:\End(D')\rightarrow\End(D)$ (defined as above) is a monoid isomorphism;
   \item $\mathscr P_D$ is an induced subposet of $\mathscr P_{D'}$.
\end{enumerate}
\end{lemma}
\begin{proof} (i) It is clear that $\Psi(\textrm{id}_{D'})=\textrm{id}_D$, and, using the fact that $\psi\vert_{V(D)}\in\mathrm{End}(D)$ for any $\psi\in\End(D')$, we see that $\Psi(\varphi\circ\psi)=(\varphi\circ\psi)\vert_{V(D)}=\varphi\circ\psi\vert_{V(D)}=\varphi\vert_{V(D)}\circ\psi\vert_{V(D)}=\Psi(\varphi)\circ\Psi(\psi)$ for any $\varphi,\psi\in\mathrm{End}(D')$.

(ii) This is straightforward, again, from the fact that $\psi\vert_{V(D)}\in\mathrm{End}(D)$ for any $\psi\in\End(D')$.
\end{proof}

We will proceed as follows. Given a monoid $M$, the first binary relational system in our sequence, $D_1$, will be the coloured Cayley graph of $M$ with respect to any set of generators; in other words, $D_1\coloneq\Cay_{\mathrm{col}}(M,C)$ with $M=\langle C\rangle$. 

\begin{remark}\label{rem:basic_Cayley} 
\begin{enumerate}[label=(\roman*)]
    \item Such a first step is standard in this type of problems, due to the well-known fact that $\mathrm{End}(D_1)\cong M$. More precisely, every endomorphism of $D_1$ corresponds to left-multiplication by an element of $M$; see for instance~\cite[Theorem~7.3.7]{Kna-19}.
    \item In this situation, the endomorphism class $[e]$ of $D_1$ is the set of left-invertible elements of $M$, the automorphism class $\llbracket e\rrbracket$ is the set of invertible elements, and since $M$ is finite $[e]=\llbracket e\rrbracket$. Also, the relation $\sim$ on $V(D_1)$ can be reformulated as $v\sim w$ if and only if $Mv=Mw$.
\end{enumerate}
\end{remark}

The next binary relational systems, $D_2$, $D_3$ and $D_4$ (four steps in total), will satisfy $D_1\leq_{\mathrm{End}} D_2\leq_{\mathrm{End}}D_3\leq_{\mathrm{End}}D_4$ (in particular, $\mathrm{End}(D_4)\cong M$, by Lemma~\ref{lem:leqend}). Moreover, $D_2$ will have the property that all the vertices in any given endomorphism class have the same total degree. In $D_3$, the total degree will be constant on every connected component of $\mathscr P_{D_3}$, and in $D_4$, the total degree will be constant on all $V(D_4)$, concluding the construction. The details of this argument can be found in Lemma~\ref{lem:step4}. Lemmas~\ref{lem:step1}, \ref{lem:step2} and \ref{lem:step3} are devoted to the intermediate constructions. 

\begin{lemma}\label{lem:step1} Let $M$ be a non-group monoid and let $C\subseteq M$ be a nonempty generating set of $M$. Let $D_1=\Cay_{\mathrm{col}}(M,C)$ and, for each $\alpha\in\mathscr A_{D_1}$, let $k_{\alpha}$ be a positive integer. Then there exists a binary relational system $D_2$ such that
\begin{enumerate}[label=(\roman*)]
   \item $D_1\leq_{\mathrm{End}}D_2$;
   \item the ground set of $\mathscr P_{D_1}$ is a connected component of $\mathscr P_{D_2}$;
   \item there is a vertex $z\in V(D_2)\backslash V(D_1)$ of total degree at least $2k_{\llbracket e\rrbracket}$ such that, for any $x,y\in V(D_2)\backslash V(D_1)\backslash\{z\}$ with $x\not\sim y$, $[x]$ and $[y]$ are incomparable and $[x]\prec [z]$;
   \item for each $v\in V(D_2)$ different from the vertex $z$ from (iii), $$\deg_{D_2}v=\begin{cases}
 \deg_{D_1}v+\sum_{w\in V(D_1)\backslash\llbracket e\rrbracket}k_{\llbracket w\rrbracket}& \text{if }v\in\llbracket e\rrbracket, \\
 \deg_{D_1}v+k_{\llbracket v\rrbracket}+\sum_{\alpha\in\mathscr A_{D_1}}k_{\alpha}+\sum_{\substack{\llbracket w\rrbracket\in\mathscr A_{D_1}, \\ [w]\preceq [v]}}k_{\llbracket w\rrbracket}& \text{if }v\in V(D_1)\backslash\llbracket e\rrbracket, \\
 |\llbracket e\rrbracket|+1 &\text{if } v\in V(D_2)\backslash V(D_1).
\end{cases}$$
\end{enumerate}
\end{lemma}
\begin{proof} 

We construct $D_2$ as follows. For each $v\in V(D_1)\backslash\llbracket e\rrbracket$ we let $P_v=\{v_1,\ldots,v_{k_{\llbracket v\rrbracket}}\}$ be a set of size $k_{\llbracket v\rrbracket}$. We assume that all these sets are pairwise disjoint, and disjoint from $V(D_1)$. We further let $z$ be a new vertex, not in $V(D_1)$ nor in any $P_v$. We define 
$$V(D_2)\coloneq V(D_1)\cup\bigcup_{v\in V(D_1)\backslash\llbracket e\rrbracket}P_v\cup\{z\}.$$ 
Next, we extend the set $C$ of the colours of the arcs of $D_1$ to a set $C'$ of colours for the arcs of $D_2$. For each $\alpha\in\mathscr A_{D_1}$ and each $1\leq i\leq k_{\alpha}$ we let $\alpha_i,\alpha'_i$ be two new colours. We can assume that all these new colours are different from each other, and not already in $C$. Thus, the set of colours of $D_2$ is 
$$C'\coloneq C\cup\bigcup_{\alpha\in\mathscr A_{D_1}}\{\alpha_1,\alpha'_1,\ldots,\alpha_{k_{\alpha}},\alpha'_{k_{\alpha}}\}.$$
Finally, the arcs of $D_2$ are defined as (see Figure~\ref{fig:step1_local})
$$A_c(D_2)\coloneq\begin{cases}
A_c(D_1) &\text{if }c\in C, \\
\{(v,z)\mid v\in V(D_1),\ [w]\preceq [v]\}\cup\{(v,v_i)\mid v\in\alpha\} &\text{if }c=\alpha_i \text{ for some } \\ 
&\alpha=\llbracket w\rrbracket\in\mathscr A_{D_1}\backslash\{\llbracket e\rrbracket\} \\
& \text{and }1\leq i\leq k_{\alpha}, \\
\{(v,z)\mid v\in V(D_1),\ v\not\approx e\}\cup\{(u,v_i)\mid v\in\alpha,\ u\approx e\} &\text{if }c=\alpha'_i \text{ for some } \\&\alpha=\llbracket w\rrbracket\in\mathscr A_{D_1}\backslash\{\llbracket e\rrbracket\} \\
& \text{and }1\leq i\leq k_{\alpha}, \\
\{(v,z)\mid v\in V(D_1),\ v\not\approx e\}&\text{if $c=\llbracket e\rrbracket_i$ or $c=\llbracket e\rrbracket'_i$} \\
& \text{for some } 1\leq i\leq k_{\llbracket e\rrbracket}.
\end{cases}
$$
From the construction it is clear that $D_1$ is an induced subsystem of $D_2$.
\begin{figure}[h]
\centering
\includegraphics[width=0.3\textwidth]{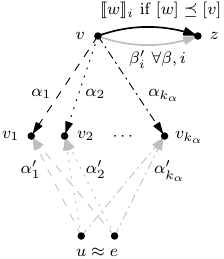}
\caption{Sketch of the neighbourhood of $\{v\}\cup P_v$ with respect to the new colours $C'\backslash C$, for a generic $v\in V(D_1)\backslash\llbracket e\rrbracket$ with automorphism class $\alpha$. The arcs of colours $\alpha_1,\ldots,\alpha_{k_{\alpha}}$ are in black and the arcs of colours $\alpha'_1,\ldots,\alpha'_{k_{\alpha}}$ are in grey. The thick black arc represents arcs of all the colours of the form $\llbracket w\rrbracket_i$ with $[e]\preceq[w]\preceq[v]$ and $1\leq i\leq k_{\llbracket w\rrbracket}$, and the thick grey arc represents arcs of all the colours of the form $\beta'_i$ with $\beta\in\mathscr A_{D_1}$ and $1\leq i\leq k_{\beta}$. Every vertex $u\in\llbracket e\rrbracket$ has the same adjacencies to $P_v$.} 
\label{fig:step1_local}
\end{figure}

Given $\varphi\in\mathrm{End}(D_1)$, let $\varphi':V(D_2)\rightarrow V(D_2)$ be the mapping defined by $$\varphi'(x)=\begin{cases}
\varphi(x) &\text{if } x\in V(D_1),\\
\varphi(v)_i &\text{if } x=v_i \text{ for some } v\in V(D_1)\backslash\llbracket e\rrbracket \text{ and } 1\leq i\leq k_{\llbracket v\rrbracket} \text{, and } \varphi\in\mathrm{Aut}(D_1), \\
z &\text{if } x=v_i \text{ for some } v\in V(D_1)\backslash\llbracket e\rrbracket \text{ and } 1\leq i\leq k_{\llbracket v\rrbracket} \text{, and } \varphi\notin\mathrm{Aut}(D_1), \\
z &\text{if } x=z.
\end{cases}$$
Note that if $\varphi\in\Aut(D_1)$ then $\varphi(v)\approx v$, so in the second case of the definition the expression $\varphi(v)_i$ makes sense. In what follows we routinely check that $\varphi'\in\mathrm{End}(D_2)$ by cases. Let $c\in C'$ be a colour and $a\in A_c(D_2)$ a $c$-arc.
\begin{itemize}
   \item[] \textit{Case 1: $c\in C$.} It is then immediate that $\varphi'(a)\in A_c(D_2)$.
   \item[] \textit{Case 2: $c=\alpha_i$ for some $\alpha=\llbracket w\rrbracket\in\mathscr A_{D_1}\backslash\{\llbracket e\rrbracket\}$ and $1\leq i\leq k_{\alpha}$.} First assume that $a$ is of the form $(v,v_i)$, where $v\in\alpha$. If $\varphi\in\mathrm{Aut}(D_1)$, then $\varphi(v)\in\alpha$, so $\varphi'(a)=(\varphi(v),\varphi(v)_i)\in A_{\alpha_i}(D_2)$. If $\varphi\notin\mathrm{Aut}(D_1)$, then $\varphi'(a)=(\varphi(v),z)$, which indeed is an $\alpha_i$-arc because $[v]\preceq[\varphi(v)]$. Now suppose that $a$ is of the form $(v,z)$, with $v\in V(D_1)$ and $[w]\preceq [v]$. Then $\varphi'(a)=(\varphi(v),z)\in A_{\alpha_i}(D_2)$, because $[w]\preceq [v]\preceq [\varphi(v)]$.
   \item[]\textit{Case 3: $c=\alpha'_i$ for some $\alpha=\llbracket w\rrbracket\in\mathscr A_{D_1}\backslash\{\llbracket e\rrbracket\}$ and $1\leq i\leq k_{\alpha}$.} First we assume that $a$ is of the form $(u,v_i)$, where $v\in\alpha$ and $u\approx e$. If $\varphi\in\mathrm{Aut}(D_1)$, then $\varphi(v)\in\alpha$ and $\varphi(u)\approx e$, so $\varphi'(a)=(\varphi(u),\varphi(v)_i)\in A_{\alpha'_i}(D_2)$. If $\varphi\notin\mathrm{Aut}(D_1)$ then, by Remark~\ref{rem:basic_Cayley}, $\varphi(u)$ is not an invertible element of $M$, i.e.~$\varphi(u)\not\approx e$, so $\varphi'(a)=(\varphi(u),z)\in A_{\alpha'_i}(D_2)$. Now suppose that $a$ is of the form $(v,z)$, with $v\in V(D_1)$ and $v\not\approx e$. Then $\varphi(v)\not\approx e$ by Remark~\ref{rem:basic_Cayley}, so $\varphi'(a)=(\varphi(v),z)\in A_{\alpha'_i}(D_2)$.
   \item[]\textit{Case 4: $c=\llbracket e\rrbracket_i$ or $c=\llbracket e\rrbracket'_i$ for some $1\leq i\leq k_{\llbracket e\rrbracket}$.} Then $a=(v,z)$ with $v$ a non-invertible element of $M$. Since $\varphi(v)$ is not invertible either, $\varphi(a)=(\varphi(v),z)\in A_c(D_2)$.
\end{itemize}

Let us see now that the mapping $\Phi:\End(D_1)\rightarrow\End(D_2)$ defined by $\Phi(\varphi)=\varphi'$ is a bijection. Clearly, $\Phi$ is injective. Let $\psi\in\End(D_2)$. Since $C\neq\emptyset$, every vertex in $V(D_1)$ is the origin of an arc of a colour in $C$, so $\psi(V(D_1))\subseteq V(D_1)$. Moreover, it is clear that $\psi\vert_{V(D_1)}\in\End(D_1)$. Consider a vertex $v_i$, with $v\in V(D_1)\backslash\llbracket e\rrbracket$ and $1\leq i\leq k_{\llbracket v\rrbracket}$. Since $v_i$ is a $\llbracket v\rrbracket_i$-out-neighbour of $v$, $\psi(v_i)$ must be either $\psi(v)_i$ (if it exists) or $z$, the only two possible $\llbracket v\rrbracket_i$-out-neighbours of $\psi(v)$. But $v_i$ is a $\llbracket v\rrbracket'_i$-out-neighbour of $e$. Therefore, if $\psi\vert_{V(D_1)}\in\Aut(D_1)$, then $\psi(e)\approx e$, and the only option is $\psi(v_i)=\psi(v)_i$. If, conversely, $\psi\vert_{V(D_1)}\notin\Aut(D_1)$, then $\psi(e)\not\approx e$, and the only option is $\psi(v_i)=z$. Finally, concerning the special vertex $z$, note that $\psi$ sends $z$ to itself because $z$ is the only possible end of $\llbracket e\rrbracket_1$-arcs. Indeed, $\llbracket e\rrbracket_1$-arcs do exist due to $M$ not being a group, and thus having some element $v\in M\backslash\llbracket e\rrbracket$ (see Remark~\ref{rem:basic_Cayley}). We conclude that $\psi=(\psi\vert_{V(D_1)})'$.

Thus, we have seen that $D_1\leq_{\End}D_2$. To prove (ii) and (iii), we use the fact that every endomorphism of $D_2$ is of the form $\varphi'$ for some endomorphism $\varphi$ of $D_1$. 
And the formula in (iv) for the total degree of the vertices follows easily from the construction.

Finally, note that when $M$ is a group the conclusion does not hold, because $D_2$ is just $D_1$ together with an isolated vertex. We do not need to worry about this exception, as, for groups, the whole construction leading to Lemma~\ref{lem:step4} is pointless. 
\end{proof}

Given a binary relational system $D_2$ such that the total degree is constant on each endomorphism class, we now want to construct a binary relational system $D_3$ such that $D_2\leq_{\End} D_3$ and the total degree is constant on each connected component of $\mathscr P_{D_3}$. We will later achieve this by recursively applying the next lemma: in the first application, $D$ will be set to be $D_2$, and the last $D'$ obtained will be $D_3$.

An \emph{ideal} of a poset $(P,\leq)$ is a subset $I\subseteq P$ with the property that $x\leq y\in I$ implies $x\in I$.

\begin{lemma}\label{lem:step2} Let $D$ be a binary relational system such that $\deg_D v=\deg_D w$ for every two vertices $v,w$ with $v\sim w$, let $I$ be an ideal of $\mathscr P_D$ such that $\deg_D v=\deg_D w$ for every two vertices $v,w$ with $[v],[w]\in I$, and let $k$ be a positive integer. Then, there exists a binary relational system $D'$ such that
\begin{enumerate}[label=(\roman*)]
   \item $D\leq_{\End} D'$;
   \item every connected component of $\mathscr P_D$ is a connected component of $\mathscr P_{D'}$;
   \item $\deg_{D'}x=\deg_{D'}y$ for every two vertices $x,y\in V(D')\backslash V(D)$;
   \item for each $v\in V(D)$, $$\deg_{D'}v=\begin{cases}
   \deg_D v+2k &\text{if }[v]\in I \\
   \deg_D v+k+1 &\text{otherwise.}\end{cases}$$
\end{enumerate}
\end{lemma}
\begin{proof} For each $v\in V(D)$, we create a set of new vertices $P_v$ of size $k$ or $1$, depending on whether $[v]\in I$ or not. More precisely, 
$$P_v=\begin{cases}
\{v_1,\ldots,v_k\} &\text{if }[v]\in I\\
\{v_1\} &\text{otherwise,}\end{cases}$$
and we suppose that all these sets are disjoint from each other, and disjoint from $V(D)$. The set of vertices of $D'$ is defined as
$$V(D')\coloneq V(D)\cup\bigcup_{v\in V(D)}P_v.$$
The set of colours $C'$ for the arcs of $D'$ is defined as
$$C'\coloneq C\cup\{0,1,\ldots,k\}\cup\{p\},$$
where $C$ is the set of colours of the arcs of $D$, and $0,1,\ldots,k,p$ are new colours (the union is assumed to be disjoint). Finally, the arcs of $D'$ are defined as (see Figure~\ref{fig:step2_local})
$$A_c(D')\coloneq\begin{cases}
A_c(D) &\text{if } c\in C \\
\{(v,v_i)\mid v\in V(D),\ 1\leq i\leq |P_v|\} &\text{if } c=0 \\
\{(v,v_i)\mid v\in V(D),\ [v]\in I\} \cup \{(v,v_1)\mid v\in V(D),\ [v]\notin I\} &\text{if } c=i\in\{1,\ldots,k\} \\
\{(v_i,v_j)\mid v\in V(D),\ 1\leq i\leq j\leq |P_v|\} &\text{if } c=p. \\
\end{cases}$$
From the construction it is clear that $D$ is an induced subsystem of $D'$.
\begin{figure}[h]
\centering
\includegraphics[width=0.65\textwidth]{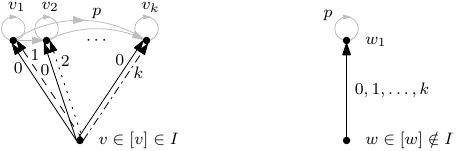}
\caption{Sketch of the neighbourhood, with respect to the new colours $C'\backslash C$, of a vertex $v\in V(D)$ with $[v]\in I$ and a vertex $w\in V(D)$ with $[w]\notin I$. The arcs of colours $0,1,\ldots,k$ are in black and the arcs of colour $p$ are in grey. The thick black arc represents arcs of all the colours $0,1,\ldots,k$.} 
\label{fig:step2_local}
\end{figure}

Given $\varphi\in\End(D)$, let $\varphi':V(D')\rightarrow V(D')$ be the mapping defined by
$$\varphi'(x)=\begin{cases}
\varphi(x) &\text{if } x\in V(D) \\
\varphi(v)_i &\text{if } x=v_i \text{ for some $v\in V(D)$ with $[\varphi(v)]\in I$} \\
\varphi(v)_1 &\text{if } x=v_i \text{ for some $v\in V(D)$ with $[\varphi(v)]\notin I$.}
\end{cases}$$
Note that $\varphi'$ is well-defined: in the second case of the definition, when $x=v_i$, $\varphi(v)_i$ is indeed a vertex of $D'$ because $[\varphi(v)]\in I$. Let us check that $\varphi'\in\End(D')$. Consider a colour $c\in C'$ and a $c$-arc $a\in A_c(D')$.
\begin{itemize}
   \item[] \textit{Case 1: $c\in C$.} It is then immediate that $\varphi'(a)\in A_c(D')$.
   \item[] \textit{Case 2: $c=0$.} Let us write $a=(v,v_i)$. If $[\varphi(v)]\in I$, then $\varphi'(a)=(\varphi(v),\varphi(v)_i)\in A_c(D')$, and if $[\varphi(v)]\notin I$ then $\varphi'(a)=(\varphi(v),\varphi(v)_1)\in A_c(D')$.
   \item[] \textit{Case 3: $c\in\{1,\ldots,k\}$.} First assume that $a=(v,v_c)$, where $[v]\in I$. Then, $[\varphi(v)]\in I$ implies that $\varphi'(a)=(\varphi(v),\varphi(v)_c)\in A_c(D')$, and $[\varphi(v)]\notin I$ implies that $\varphi'(a)=(\varphi(v),\varphi(v)_1)\in A_c(D')$. Now assume that $a=(v,v_1)$, where $[v]\notin I$. Then, $[\varphi(v)]\notin I$, because $[v]\preceq [\varphi(v)]$ and $I$ is an ideal. Therefore $\varphi'(a)=(\varphi(v),\varphi(v)_1)\in A_c(D')$.
   \item[] \textit{Case 4: $c=p$.} Let $a=(v_i,v_j)$ with $v\in V(D)$ and $1\leq i\leq j\leq |P_v|$. If $[\varphi(v)]\in I$ then $\varphi'(a)=(\varphi(v)_i,\varphi(v)_j)\in A_c(D')$, and if $[\varphi(v)]\notin I$ then $\varphi'(a)=(\varphi(v)_1,\varphi(v)_1)\in A_c(D')$.
\end{itemize}

Let us see now that the mapping $\Phi:\End(D)\rightarrow\End(D')$ defined by $\Phi(\varphi)=\varphi'$ is a bijection. Clearly, $\Phi$ is injective. Let $\psi\in\End(D')$. All the vertices in $V(D)$ are the origin of a $0$-arc, so $\psi(V(D))\subseteq V(D)$. Moreover, it is clear that $\psi\vert_{V(D)}\in\End(D)$, and, thanks to the colours $\{1,\ldots,k\}\subseteq C'$, the images by $\psi$ of the vertices in $V(D')\backslash V(D)$ are determined by $\psi\vert_{V(D)}$. Therefore, $\psi=(\psi\vert_{V(D)})'$, as we wanted to see. 

This shows that $D\leq_{\End} D'$. From the description of the endomorphisms of $D'$, it follows that any $\pi\in\mathscr P_D$ and $\rho\in\mathscr P_{D'}\backslash\mathscr P_D$ are incomparable, showing (ii). Lastly, (iii) and (iv) can be easily checked.
\end{proof}

In the last step, we start with a binary relational system $D_3$ where the total degree is constant in each connected component of $\mathscr P_{D_3}$, and we build a binary relational system $D_4$ with $D_3\leq_{\End} D_4$ where the total degree is constant. This can be achieved by  taking products with an arc of multiple colours, component by component, in a suitable way.

\begin{lemma}\label{lem:step3} Let $D_3$ be a binary relational system such that, for every two vertices $v,w$ such that $[v]$ and $[w]$ belong to the same connected component of $\mathscr P_{D_3}$, $v$ and $w$ have the same total degree. 
Then, there exists a binary relational system $D_4$ such that
$D_3\leq_{\End} D_4$, and the total degree of every vertex of $D_4$ is the same.
\end{lemma}
\begin{proof} For each $v\in V(D_3)$, let $v'$ be a new vertex. The vertex set of $D_4$ is $V(D_4)\coloneq V(D_3)\cup\{v'\mid v\in V(D_3)\}.$ Let $\Delta$ and $\delta$ be the maximum and the minimum total degree of $D_3$, respectively. The set of colours of the arcs of $D_4$ is 
$C'\coloneq C\cup\{\delta,\delta+1,\ldots,\Delta\},$
where the union is supposed to be disjoint. The arcs of $D_4$ are 
$$A_c(D_4)\coloneq\begin{cases}
\{(v,w),(v',w')\mid (v,w)\in A_c(D_3)\}&\text{if $c\in C$} \\
\{(v,v')\mid v\in V(D_3),\ \deg_{D_3} v\leq c\} &\text{if } c\in\{\delta,\delta+1,\ldots,\Delta\}.
\end{cases}$$
Every vertex has total degree $\Delta+1$ in $D_4$, and $D_3$ is an induced subsystem of $D_4$. Consider the mapping $\Phi:\End(D_3)\rightarrow\End(D_4)$ defined by $\Phi(\varphi)=\varphi'$, where 
$$\varphi'(x)=\begin{cases}
\varphi(x) &\text{if }x\in V(D_3) \\
\varphi(v)' &\text{if $x=v'$ for a $v\in V(D_3)$.}
\end{cases}$$
For $v\in V(D_3)$ and $\varphi\in\End(D_3)$, $[v]$ and $[\varphi(v)]$ belong to the same connected component of $\mathscr P_{D_3}$, so $\deg_{D_3} v=\deg_{D_3} \varphi(v)$. Thus $\varphi'$ preserves the arcs of colours $\delta,\delta+1,\ldots,\Delta$, and this implies that $\Phi$ is well-defined. Let us see that it is surjective. Given $\psi\in\End(D_4)$, since the vertices in $V(D_3)$ are precisely the origins of the $\Delta$-arcs, $\psi(V(D_3))\subseteq V(D_3)$. Moreover, $\psi\vert_{V(D_3)}\in\End(D_3)$, and, again because of the $\Delta$-arcs, the images by $\psi$ of the vertices in $V(D_4)\backslash V(D_3)$ are determined by $\psi\vert_{V(D_3)}$. Thus, $\psi=(\psi\vert_{V(D_3)})'$. Clearly, $\Phi$ is also injective, so $D_3\leq_{\End} D_4$.
\end{proof}

We say that a binary relational system $D$ has the \emph{increasing degree property} if $\deg_D x\leq\deg_D y$ for any $x,y\in V(D)$ with $[x]\preceq [y]$. We are now ready to prove the main result of this section. 

\begin{lemma}\label{lem:step4} Every monoid $M$ is isomorphic to the endomorphism monoid of a binary relational system in which the total degree of every vertex is the same.
\end{lemma}
\begin{proof} Let $C$ be a nonempty generating set of $M$, and let $D_1\coloneq\Cay_{\mathrm{col}}(M,C)$. By Remark~\ref{rem:basic_Cayley}, $\End(D_1)\cong M$. In particular, if $M$ is a group, we are already done. If $M$ is not a group, using Lemma~\ref{lem:step1} we obtain a binary relational system $D_2$ such that $D_1\leq_{\End} D_2$. By (ii) of Lemma~\ref{lem:leqend}, each endomorphism class of $D_1$ is an endomorphism class of $D_2$. By (iii)--(iv) of Lemma~\ref{lem:step1} and the fact that  $[e]=\llbracket e\rrbracket$ (see Remark~\ref{rem:basic_Cayley}), the integers $k_{\alpha}$ can be chosen so that in $D_2$ the total degree is constant on each endomorphism class (i.e.~on each element of $\mathscr P_{D_2}$).  In fact, if the $k_{\alpha}$'s are chosen carefully (e.g.
\[k_{\llbracket v\rrbracket}=\begin{cases}
|\llbracket e\rrbracket|-\deg_{D_1} v+\max\{\deg_{D_1} w\mid w\in V(D_1),\ [w]\preceq [v]\} & \text{if $\llbracket v\rrbracket\neq\llbracket e\rrbracket$} \\
\sum_{w\in V(D_1)\setminus\llbracket e\rrbracket} k_{\llbracket w\rrbracket} & \text{if $\llbracket v\rrbracket=\llbracket e\rrbracket$),}
\end{cases}\]
it can be ensured that 
$D_2$ has the increasing degree property.

Now we will construct a sequence $D_2=D_2^0\leq_{\End}D_2^1\leq_{\End}\ldots\leq_{\End}D_2^t=D_3$ of binary relational systems preserving the increasing degree property that are increasingly regular (in a sense that will be made precise later), to the point where $\deg_{D_3} x=\deg_{D_3} y$ for any $x,y\in V(D_3)$ such that $[x]$ and $[y]$ belong to the same connected component of $\mathscr P_{D_3}$. Let $i\geq 0$ and assume that we already have such $D^0_2,D_2^1,\ldots,D_2^i$, but that we still cannot take $D_2^i$ as $D_3$. By the increasing degree property, we can use the notation $\deg_{D_2^i}[x]\coloneq\deg_{D_2^i} x$. Let $K_0\subseteq\mathscr P_{D_2^i}$ be a connected component where the total degree is not constant, and let  $I\subseteq K_0$ be the set of endomorphism classes $\rho$ minimizing $\deg_{D_2^i}\rho$ on $K_0$. Denote by $\deg_{D_2^i}I$ this minimum. By the increasing degree property, $I$ is an ideal. Let $\pi\in K_0\backslash I$ such that $I\cup\{\pi\}$ is an ideal, minimizing $\deg_{D_2^i}\pi$. We construct $D_2^{i+1}$ from $D_2^i$ by applying Lemma~\ref{lem:step2} with $I$ and $k=\deg_{D_2^i}\pi-\deg_{D_2^i} I+1$. Observe that
\begin{enumerate}[label=(\arabic*)]
   \item $D_2^{i+1}$ has the increasing degree property;
   \item every connected component of $\mathscr P_{D_2^i}$ is a connected component of $\mathscr P_{D_2^{i+1}}$;
   \item if $\deg_{D_2^i}$ is constant on $\cup_{\alpha\in K}\,\alpha$ for  a connected component $K$ of $\mathscr P_{D_2^i}$, then $\deg_{D_2^{i+1}}$ is also constant on $\cup_{\alpha\in K}\,\alpha$;
   \item $\deg_{D_2^{i+1}}$ is constant on $\cup_{\alpha\in K}\,\alpha$ for each connected component $K$ of $\mathscr P_{D_2^{i+1}}\backslash\mathscr P_{D_2^i}$; 
   \item the set of endomorphism classes $\rho$ minimizing $\deg_{D_2^{i+1}}\rho$ on $K_0$ includes $I\cup\{\pi\}$.   
\end{enumerate}
By applying the same argument sufficiently many times, say $j$, in the resulting binary relational system $D_2^{i+j}$ the total degree will be constant on the connected component $K_0$ of $\mathscr P_{D_2^{i+j}}$, while remaining constant on the connected components where $\deg_{D_2^i}$ was already constant. Now, we can repeat this whole process for every connected component of $\mathscr P_{D_2}$ in which the total degree is not constant, and at the end we will obtain the desired $D_3$.

Finally, we apply Lemma~\ref{lem:step3} to $D_3$, and we obtain a binary relational system $D_4$ such that $\End(D_4)\cong M$ and the total degree of every vertex is the same. 
\end{proof}

\subsection{Representing monoids with regular graphs}\label{sec:babaipultr}

\babaipultr*
\begin{proof} By Lemma~\ref{lem:step4}, there is a binary relational system $D^*$ with $\End(D^*)\cong M$, of minimum indegree and outdegree at least $1$ (if necessary, add loops of a new colour to every vertex), such that every vertex has the same total degree $d\geq 3$. Now, note that (any finite subset of) the infinite family $\mathcal F_2(d)$ from Remark~\ref{rem:sausages}, viewed as a tuple of oriented $d$-indicators, satisfies the hypotheses of Lemma~\ref{lem:sip_product_and_categories}. Therefore, using the second version of the \v s\'ip product (the one with symbol $\vec *$), we can replace the coloured arcs of $D^*$ by appropriate gadgets and obtain a digraph $D$ with $\End(D)\cong M$, of minimum indegree and outdegree at least $1$, and such that every vertex of $D$ has total degree $d$. Let $S(d,g)$ be the $d$-indicator of odd girth $g\geq 7$ obtained with Claim~\ref{claim:strong_induction}. By applying Lemma~\ref{lem:sip_product_and_categories} again, we see that the graph $G\coloneq D*S(d,g)$ is regular of degree $d$, has odd girth $g$ (see Observation~\ref{obs:odd_girth_sip_product}), and satisfies $\End(G)\cong M$. We can obtain more instances of $G$ by using larger gadgets.
\end{proof}

We note that the constructed graphs have the following additional property.

\begin{obs} Each graph $G$ exhibited in the proof of Theorem~\ref{teo:babaipultr} admits an orientation $\vec G$ with $\End(\vec G)\cong\End(G)\cong M$.
\end{obs}
\begin{proof}
Let $G=D\mathbin{*}S(d,g)$ as in the proof of Theorem~\ref{teo:babaipultr}, and let $\vec S(d,g)$ be an orientation of $S(d,g)$. Since $S(d,g)$ satisfies the hypotheses of Lemma~\ref{lem:sip_product_and_categories}, $\vec S(d,g)$ satisfies them too (note in particular that we do not need $\vec S(d,g)$ to be strongly connected; the connectivity of its underlying undirected graph $S(d,g)$ is enough). Thus, by Lemma~\ref{lem:sip_product_and_categories}, the orientation $\vec G\coloneq D\mathbin{\vec *}\vec S(d,g)$ of $G$ has $\End(\vec G)\cong\End(D)\cong M$.
\end{proof}

\section{Discussion}

Theorem~\ref{thm:regular_rigid_family} says that there are many $d$-regular rigid graphs for every $d\geq 3$. Another notion of \emph{many} corresponds to an intriguing suspicion raised by Petrov in the comments to~\cite{vdZ19}: are random $d$-regular graphs rigid with probability almost $1$? It is known~\cite{Bollobas82} that asymptotically almost all $d$-regular graphs on $n$ vertices are asymmetric ($d\geq 3$).

Theorem~\ref{thm:regular_rigid_family} implies that for any $d\geq 3$ there exists a $d$-regular rigid graph of girth $\geq 4$. 
Do rigid $d$-regular graphs of girth at least $g$ exist for any pair $(d,g)$? 
Note that, in the case of asymmetric graphs, Bollob\'as' result~\cite{Bollobas82} implies an affirmative answer (and, in the rigid case, a positive answer to the question in the previous paragraph would suffice, see \cite[Theorem 9.5]{JLR2000}).
If yes, what is the smallest order of such a graph? 
We know the precise answer only for all $(d,3)$ thanks to Theorem~\ref{thm:smallrigid} and~\cite{BI69}. Combining computation with some results from the literature on cages~\cite{EJ13}, we compiled  Table~\ref{table}. In~\cite{arXiv:2511.07247} new upper bounds were constructed and the graphs for $(8,7)$, $(10,7)$, $(12,7)$, and $(14,7)$ are asymmetric. 

\setlength{\tabcolsep}{4.2pt} 
\begin{table}[htbp]
    \centering
{\footnotesize
    \begin{tabular}{c|ccc|ccc|ccc|ccc|ccc}
        $d$ & ~ & $3$ & ~ & ~ & $4$ & ~ & ~ & $5$ & ~ & ~ & $6$ & ~ & ~ & $7$ & ~ \\ 
        \hline
        $g$ & cage & asym & rigid & cage & asym & rigid & cage & asym & rigid & cage & asym & rigid & cage & asym & rigid \\ \hline
        3 & 4 & 12 & 14 & 5 & 10 & 10 & 6 & 10 & 10 & 7 & 11 & 11 & 8 & 12 & 12 \\ \hline
        4 & 6 & 14 & 14 & 8 & 13 & 13 & 10 & 16 & 16 & 12 & 18 & 19 & 14 & 20 & $\leq 24$ \\ \hline
        5 & 10 & 16 & 16 & 19 & 22 & 22 & 30 & 32 & 32 & 40 & $\geq41$ &  & 50 & $\geq52$ &  \\ \hline
        6 & 14 & 22 & 22 & 26 & 34 & $\geq36$ & 42 &  &  & 62 &  &  & 90 &  &  \\ \hline
        7 & 24 & 28 & 28 & 67 & &  &  &  &  &  &  &  &  &  &  \\ \hline
        8 & 30 & 40 &  &  &  &  &  &  &  &  &  &  &  &  &  \\ \hline
        9 & 58 & 58 &  &  &  &  &  &  &  &  &  &  &  &  &  \\    \end{tabular}
    }
    
    \caption{Orders of cages and smallest regular rigid and asymmetric graphs.}
    \label{table}
\end{table}

It follows from a result of Roberson~\cite{Rob19} that strongly regular graphs are \emph{core-complete}, i.e., all the endomorphisms of a strongly regular graph are automorphisms or it has an endomorphism whose image is a complete graph. In particular, strongly regular graphs are not able to represent all monoids. It is open whether more generally all distance regular graphs are core-complete~\cite[Section 16.4]{vKT16}. To the contrary it could be true that distance regular graphs are even able to represent all monoids. We do not know.

As mentioned in the introduction, the class of (completely regular) monoids cannot be represented by graphs of bounded degree~\cite{BP80,knauer2023endomorphismuniversalitysparsegraph}. However, semilattices, as well as $k$-cancellative monoids, can be represented by graphs of bounded degree~\cite{knauer2023endomorphismuniversalitysparsegraph}. Can $k$-cancellative monoids be represented by $d$-regular graphs, where $d$ depends on $k$? Note that for $k=1$ the answer is $d=3$, due to the result of Hell and Ne\v{s}et\v{r}il~\cite{HN73}. More generally, is every endomorphism monoid of a graph of maximum degree at most $k$ isomorphic to the endomorphism monoid of a $d$-regular graph, for some $d$ depending on $k$?

\paragraph{AI-statement:} No AI was used at any stage of the production of this manuscript.

\paragraph{Acknowledgments:} GPS is grateful to Guillem Perarnau and \'Erika Rold\'an for their advice. KK and GPS were supported through grant 
PID2022-137283NB-C22 funded by MICIU/AEI/10.13039/501100011033 and by ERDF/EU and through the Severo Ochoa and Mar\'ia de Maeztu Program for Centers and Units of Excellence in R\&D (CEX2020-001084-M) and ANR-21-CE48-0012. The authors are also thankful to the referees for the careful reading and very helpful comments.

\small

\bibliography{lit}

@Book{AHS2004,
 Author = {Ji\v{r}\'i Ad\'amek and Horst Herrlich and George E. Strecker},
 Title = {{Abstract and Concrete Categories}},
 Year = {2004},
 Publisher = {Online Edition}
}

@Book{Bing1983,
 Author = {R. H. Bing},
 Title = {The geometric topology of $3$-manifolds},
 Series = {{Colloq. Publ.}},
 Volume = {40},
 Year = {1983},
 Publisher = {American Math. Soc.}
}

@Article{HN73,
 Author = {Pavol Hell and Jaroslav Ne\v{s}et\v{r}il},
 Title = {Groups and monoids of regular graphs (and of graphs with bounded degrees)},
 Journal = {Can. J. Math.},
 Volume = {25},
 Number = {2},
 Pages = {239--251},
 Year = {1973}
}

@article{Frucht49, 
 title={Graphs of Degree Three with a Given Abstract Group}, 
 volume={1}, 
 DOI={10.4153/CJM-1949-033-6}, 
 number={4}, 
 journal={Can. J. Math.}, 
 publisher={Cambridge University Press}, 
 author={Frucht, Robert}, 
 year={1949}, 
 pages={365--378}
}

@Book{MoharThomassen2001,
 Author = {B. Mohar and C. Thomassen},
 Title = {Graphs on surfaces},
 Series = {Johns Hopkins Stud. Math. Sci.},
 Year = {2001},
 Publisher = {Johns Hopkins University Press}
}

@Article{BP80,
 Author = {Laszlo {Babai} and Ales {Pultr}},
 Title = {{Endomorphism monoids and topological subgraphs of graphs}},
 FJournal = {{Journal of Combinatorial Theory. Series B}},
 Journal = {{J. Comb. Theory, Ser. B}},
 ISSN = {0095-8956},
 Volume = {28},
 Pages = {278--283},
 Year = {1980},
 Publisher = {Elsevier (Academic Press), San Diego, CA},
 Language = {English},
 MSC2010 = {05C25},
 Zbl = {0437.05029}
}

@article{Bollobas82,
author = {Bollob\'as, B\'ela},
title = {The Asymptotic Number of Unlabelled Regular Graphs},
journal = {J. London Math. Soc.},
volume = {s2-26},
number = {2},
pages = {201--206},
doi = {https://doi.org/10.1112/jlms/s2-26.2.201},
url = {https://londmathsoc.onlinelibrary.wiley.com/doi/abs/10.1112/jlms/s2-26.2.201},
eprint = {https://londmathsoc.onlinelibrary.wiley.com/doi/pdf/10.1112/jlms/s2-26.2.201},
year = {1982}
}

@Book{JLR2000,
 Author = {{S. Janson, T. \L uczak, and A. Ruci\'nski}},
 Title = {{Random Graphs}},
 publisher = {{John Wiley \& Sons}},
 Year = {2000}
}

@article{Frucht1939,
     author = {Frucht, Robert},
     title = {{Herstellung von Graphen mit vorgegebener abstrakter Gruppe}},
     journal = {Compos. Math.},
     pages = {239--250},
     publisher = {Johnson Reprint Corporation},
     volume = {6},
     year = {1939},
     zbl = {0020.07804},
     language = {de},
     url = {http://www.numdam.org/item/CM_1939__6__239_0/}
}

@MISC{vdZ19,
    TITLE = {Strongly rigid regular graphs},
    AUTHOR = {Dominic van der Zypen},
    HOWPUBLISHED = {MathOverflow},
    NOTE = {\texttt{https://mathoverflow. net/q/321108} (version: 2019-01-18)},
    EPRINT = {https://mathoverflow.net/q/321108},
    URL = {https://mathoverflow.net/q/321108}
}

@MISC {vdZ18,
    TITLE = {Strongly rigid connected $k$-regular graphs},
    AUTHOR = {Dominic van der Zypen},
    HOWPUBLISHED = {MathOverflow},
    NOTE = {\texttt{https:// mathoverflow.net/q/296483} (version: 2018-03-29)},
    EPRINT = {https://mathoverflow.net/q/296483},
    URL = {https://mathoverflow.net/q/296483}
}

@Article{BI69,
 Author = {Baron, G. and Imrich, W.},
 Title = {Asymmetrische regul{\"a}re {Graphen}},
 FJournal = {Acta Mathematica Academiae Scientiarum Hungaricae},
 Journal = {Acta Math. Acad. Sci. Hung.},
 ISSN = {0001-5954},
 Volume = {20},
 Pages = {135--142},
 Year = {1969},
 Language = {German},
 DOI = {10.1007/BF01894574},
 zbMATH = {3303827},
 Zbl = {0191.22604}
}

@Book{HIK11,
 Author = {Hammack, Richard and Imrich, Wilfried and Klav{\v{z}}ar, Sand},
 Title = {Handbook of product graphs},
 Edition = {2nd},
 FSeries = {Discrete Mathematics and its Applications},
 Series = {Discrete Math. Appl. (Boca Raton)},
 ISBN = {978-1-4398-1304-1; 978-1-138-19908-8; 978-1-4398-1305-8},
 Year = {2011},
 Publisher = {Boca Raton, FL: CRC Press},
 Language = {English},
 URL = {www.crcnetbase.com/isbn/9781439813058},
 zbMATH = {5906416},
 Zbl = {1283.05001}
}

@Article{Izb60,
 Author = {Izbicki, Herbert},
 Title = {Regul{\"a}re {Graphen} beliebigen {Grades} mit vorgegebenen {Eigenschaften}},
 FJournal = {Monatshefte f{\"u}r Mathematik},
 Journal = {Monatsh. Math.},
 ISSN = {0026-9255},
 Volume = {64},
 Pages = {15--21},
 Year = {1960},
 Language = {German},
 DOI = {10.1007/BF01319048},
 URL = {https://eudml.org/doc/177083},
 zbMATH = {3149614},
 Zbl = {0091.37602}
}

@Book{PT80,
 Author = {Pultr, Ales and Trnkov{\'a}, V{\v{e}}ra},
 Title = {Combinatorial, algebraic and topological representations of groups, semigroups and categories},
 FSeries = {North-Holland Mathematical Library},
 Series = {North-Holland Math. Libr.},
 Volume = {22},
 Year = {1980},
 Publisher = {Elsevier (North-Holland), Amsterdam},
 Language = {English},
 zbMATH = {3650785},
 Zbl = {0418.18004}
}

@Book{HN04,
 Author = {Hell, Pavol and Ne{\v{s}}et{\v{r}}il, Jaroslav},
 Title = {Graphs and homomorphisms},
 FSeries = {Oxford Lecture Series in Mathematics and its Applications},
 Series = {Oxf. Lect. Ser. Math. Appl.},
 Volume = {28},
 ISBN = {0-19-852817-5},
 Year = {2004},
 Publisher = {Oxford: Oxford University Press},
 Language = {English},
 zbMATH = {2117181},
 Zbl = {1062.05139}
}

@Article{S57,
 Author = {Sabidussi, Gert},
 Title = {Graphs with given group and given graph-theoretical properties},
 FJournal = {Canadian Journal of Mathematics},
 Journal = {Can. J. Math.},
 ISSN = {0008-414X},
 Volume = {9},
 Pages = {515--525},
 Year = {1957},
 Language = {English},
 DOI = {10.4153/CJM-1957-060-7},
 zbMATH = {3130758},
 Zbl = {0079.39202}
}

@Article{HP65,
 Author = {Z. {Hedrl{\'i}n} and A. {Pultr}},
 Title = {{Symmetric relations (undirected graphs) with given semigroups}},
 FJournal = {{Monatshefte f\"ur Mathematik}},
 Journal = {{Monatsh. Math.}},
 ISSN = {0026-9255},
 Volume = {69},
 Pages = {318--322},
 Year = {1965},
 Publisher = {Springer, Vienna},
 Language = {English},
 DOI = {10.1007/BF01297617},
 Zbl = {0139.24803}
}

@Article{HP64,
 Author = {Z. {Hedrl{\'i}n} and A. {Pultr}},
 Title = {{Relations (graphs) with given finitely generated semigroups}},
 FJournal = {{Monatshefte f\"ur Mathematik}},
 Journal = {{Monatsh. Math.}},
 ISSN = {0026-9255},
 Volume = {68},
 Pages = {213--217},
 Year = {1964},
 Publisher = {Springer, Vienna},
 Language = {English},
 DOI = {10.1007/BF01298508},
 Zbl = {0126.26602}
}

@Article{Rob19,
 Author = {Roberson, David E.},
 Title = {Homomorphisms of strongly regular graphs},
 FJournal = {Algebraic Combinatorics},
 Journal = {Algebr. Comb.},
 ISSN = {2589-5486},
 Volume = {2},
 Number = {4},
 Pages = {481--497},
 Year = {2019},
 Language = {English},
 DOI = {10.5802/alco.50},
 zbMATH = {7089133},
 Zbl = {1417.05252}
}

@Misc{Bab80,
 Author = {Babai, Laszlo},
 Title = {Almost all {Steiner} triple systems are asymmetric},
 Language = {English},
 HowPublished = {Ann. {Discrete} {Math}. 7 (1980), pp.~37--39},
 DOI = {10.1016/S0167-5060(08)70169-9},
 zbMATH = {3722656},
 Zbl = {0462.05013}
}

@Article{GR11,
 Author = {Godsil, Chris and Royle, Gordon F.},
 Title = {Cores of geometric graphs},
 FJournal = {Annals of Combinatorics},
 Journal = {Ann. Comb.},
 ISSN = {0218-0006},
 Volume = {15},
 Number = {2},
 Pages = {267--276},
 Year = {2011},
 Language = {English},
 DOI = {10.1007/s00026-011-0094-5},
 zbMATH = {5999779},
 Zbl = {1234.05165}
}

@manual{sagemath,
  Key          = {SageMath},
  Author       = {{The Sage Developers}},
  Title        = {{S}ageMath, the {S}age {M}athematics {S}oftware {S}ystem ({V}ersion 9.7)},
  note         = {{\tt https://www.sagemath.org}},
  Year         = {2022},
}

@Article{Mer99,
 Author = {Meringer, Markus},
 Title = {Fast generation of regular graphs and construction of cages},
 FJournal = {Journal of Graph Theory},
 Journal = {J. Graph Theory},
 ISSN = {0364-9024},
 Volume = {30},
 Number = {2},
 Pages = {137--146},
 Year = {1999},
 Language = {English},
 DOI = {10.1002/(SICI)1097-0118(199902)30:2<137::AID-JGT7>3.0.CO;2-G},
 zbMATH = {1251575},
 Zbl = {0918.05062}
}

@Article{HN92,
 Author = {Hell, Pavol and Ne{\v{s}}et{\v{r}}il, Jaroslav},
 Title = {The core of a graph},
 FJournal = {Discrete Mathematics},
 Journal = {Discrete Math.},
 ISSN = {0012-365X},
 Volume = {109},
 Number = {1-3},
 Pages = {117--126},
 Year = {1992},
 Language = {English},
 DOI = {10.1016/0012-365X(92)90282-K},
 zbMATH = {428130},
 Zbl = {0803.68080}
}

@Article{BFKS79,
 Author = {Babai, L{\'a}szl{\'o} and Frankl, P{\'e}ter and Koll{\'a}r, J{\'a}nos and Sabidussi, Gert},
 Title = {Hamiltonian cubic graphs and centralizers of involutions},
 FJournal = {Canadian Journal of Mathematics},
 Journal = {Can. J. Math.},
 ISSN = {0008-414X},
 Volume = {31},
 Pages = {458--464},
 Year = {1979},
 Language = {English},
 DOI = {10.4153/CJM-1979-051-8},
 zbMATH = {3668661},
 Zbl = {0429.05046}
}

@Article{EJ13,
 Author = {Exoo, Geoffrey and Jajcay, Robert},
 Title = {
Dynamic Cage Survey},
 FJournal = {The Electronic Journal of Combinatorics},
 Journal = {Electron. J. Comb.},
 ISSN = {1077-8926},
 Volume = {DS16},
 Pages = {55},
 Year = {2013}}

@Article{vKT16,
 Author = {van Dam, Edwin R. and Koolen, Jack H. and Tanaka, Hajime},
 Title = {Distance-regular graphs},
 FJournal = {The Electronic Journal of Combinatorics},
 Journal = {Electron. J. Comb.},
 ISSN = {1077-8926},
 Volume = {DS22},
 Pages = {156},
 Year = {2016},
 Language = {English},
 URL = {www.combinatorics.org/ojs/index.php/eljc/article/view/DS22},
 zbMATH = {6579055},
 Zbl = {1335.05062}
}

@misc{arXiv:2511.07247,
 author = {Exoo, Geoffrey and Goedgebeur, Jan and Jooken, Jorik and Stubbe, Louis and Van den Eede, Tibo},
 title = {New small regular graphs of given girth: the cage problem and beyond},
 year = {2025},
 howpublished = {Preprint, {arXiv}:2511.07247 [math.{CO}] (2025)},
 url = {https://arxiv.org/abs/2511.07247},
 arXiv = {arXiv:2511.07247}
}

@MISC {God19,
    TITLE = {Strongly rigid regular graphs},
    AUTHOR = {Chris Godsil},
    HOWPUBLISHED = {MathOverflow},
    NOTE = {\texttt{https://mathoverflow.net/ q/321225} (version: 2019-01-19)},
    EPRINT = {https://mathoverflow.net/q/321225},
    URL = {https://mathoverflow.net/q/321225}
}

@article{knauer2023endomorphismuniversalitysparsegraph,
 author = {Knauer, Kolja and Puig i. Surroca, Gil},
 title = {On endomorphism universality of sparse graph classes},
 fjournal = {Journal of Graph Theory},
 journal = {J. Graph Theory},
 issn = {0364-9024},
 volume = {110},
 number = {2},
 pages = {223--244},
 year = {2025},
 language = {English},
 doi = {10.1002/jgt.23262},
 zbMATH = {8084523}
}

@misc{BCCS76,
 author = {Bussemaker, F. C. and {\v C}obelji\'c, S. and Cvetkovi\'c, D. M. and Seidel, J. J.},
 title = {Computer investigation of cubic graphs},
 year = {1976},
 language = {English},
 howpublished = {T.H.-Report 76-{WSK}-{Ol}. {Eindhoven}, {Netherlands}: {Technological} {University}, {Department} of {Mathematics}. 66 p.},
 zbMATH = {3503313},
 Zbl = {0322.05144}
}

@Book{Kna-19,
 Author = {Knauer, Ulrich and Knauer, Kolja},
 Title = {Algebraic graph theory. {Morphisms}, monoids and matrices},
 Edition = {2nd revised and extended},
 FSeries = {De Gruyter Studies in Mathematics},
 Series = {De Gruyter Stud. Math.},
 ISSN = {0179-0986},
 Volume = {41},
 ISBN = {978-3-11-061612-5; 978-3-11-061736-8},
 Year = {2019},
 Publisher = {Berlin: De Gruyter},
 Language = {English},
 DOI = {10.1515/9783110617368},
 zbMATH = {7056043},
 Zbl = {1425.05003}
}

@article{HN05,
title = {Universal partial order represented by means of oriented trees and other simple graphs},
journal = {Eur. J. Comb.},
volume = {26},
number = {5},
pages = {765-778},
year = {2005},
issn = {0195-6698},
doi = {https://doi.org/10.1016/j.ejc.2004.01.008},
url = {https://www.sciencedirect.com/science/article/pii/S0195669804000502},
author = {Jan Hubi{\v c}ka and Jaroslav Ne{\v s}et{\v r}il},
}
\bibliographystyle{my-siam}

\end{document}